\documentclass[11pt]{article}
\usepackage[all]{xy}
\usepackage{cite}
\usepackage{graphicx}
\graphicspath{ {images/} }
\usepackage{amssymb,latexsym,amsmath,amsthm, mathrsfs}
\usepackage{algorithm}
\usepackage[noend]{algpseudocode}
\usepackage{fullpage}
\usepackage{indentfirst}
\usepackage{setspace}
\usepackage{enumitem}
\usepackage{mathtools}
\usepackage{esvect}
\usepackage[font=itshape]{quoting}

\newcommand{\F}{\mathbb F}
\newcommand{\R}{\mathbb R}
\newcommand{\D}{\mathbb D}

\newcommand{\N}{\mathbb N}
\newcommand{\I}{\mathbb I}

\def\low{{\text{\upshape low}}}
\def\up{{\text{\upshape up}}}
\def\conv{{\text{\upshape conv}}}
\def\dgm{{\text{\upshape dgm}}}
\def\INT{{\text{\upshape INT}}}
\def\dINT{\mathrm{d_{INT}}}
\def\Var{{\text{\upshape Var}}}
\def\ord{{\text{\upshape ord}}}

\def\argmin{{\text{\upshape argmin}}}
\def\Dow{{\text{\upshape Dow}}}
\def\DDow{\D{\text{\upshape ow}}}
\def\conv{{\text{\upshape conv}}}

\def\rep{{\text{\upshape rep}}}
\def\cone{{\text{\upshape cone}}}

\def\sp{{\text{\upshape span}}}

\def\interior{{\text{\upshape int}}}
\def\bd{{\text{\upshape bd}}}

\def\ord{{\text{\upshape ord}}}
\def\nerve{{\text{\upshape nerve}}}
\def\Cent{{\text{\upshape Cent}}}
\def\rank{{\text{\upshape rank}}}
\def\dist{{\text{\upshape dist}}}
\def\db{\mathrm{d_b}}
\def\di{\mathrm{d_i}}
\def\dH{\mathrm{d_H}}
\def\Unif{\mathrm{Unif}}
\def\PersistenceIntervals{{\text{\upshape PersistenceIntervals}}}
\def\ray{{\text{\upshape ray}}}

\def\RegPair{{(\mathcal{F},P_K)}} %% RegPair = regular pair

\def\od{\stackrel{\mathrm{def}}{=}}

\newtheorem{theorem}{Theorem}[section]
\newtheorem*{theorem*}{Theorem}
\newtheorem*{lemma*}{Lemma}
\newtheorem{lemma}[theorem]{Lemma}

\newtheorem{corollary}[theorem]{Corollary}

\theoremstyle{definition}\newtheorem{definition}[theorem]{Definition}
\theoremstyle{definition}\newtheorem{example}[theorem]{Example}
\theoremstyle{definition}
\theoremstyle{definition}

\title{A Topological Approach to Inferring the Intrinsic Dimension of Convex Sensing Data}
\author{Min-Chun Wu \and Vladimir Itskov}
\begin{document}
\maketitle

%%%%%%%%%%%%%%%%%%%%%%%%%%%%%
\begin{abstract}
We consider a common measurement paradigm, where an unknown subset of an affine space is measured by unknown continuous quasi-convex functions. Given the measurement data, can one determine the dimension of this space?  In this paper,  we develop a method for inferring the intrinsic  dimension of the data from measurements by quasi-convex functions, under  natural generic  assumptions. 

The dimension inference problem depends only on  discrete data of the  ordering of the measured points of space, induced  by the sensor functions.  We introduce a construction of a filtration of Dowker complexes, associated to measurements by quasi-convex functions. Topological features of these complexes are then used to infer the intrinsic dimension. We prove convergence theorems that guarantee obtaining the correct intrinsic dimension  in the limit of large data, under natural generic assumptions. We also illustrate the usability  of this method in simulations. 
\end{abstract}
\begin{spacing}{1.2}
\tableofcontents
\section{Introduction}
\label{sec:intro}
Data in many scientific applications  are often obtained by ``sensing''  the phase space via sensors/functions that are convex.  Convex sensing is a class of problems of inferring the geometry of   data that are sampled via such functions. To be precise, let us recall the following

\begin{definition}\label{def:quasi-convex}
Let $K\subseteq\R^d$ be open convex. A function $f:K\to\R$ is {\bf quasi-convex}  if each sublevel set $f^{-1}(-\infty,\ell)=\{ x \in K \, \vert\,  f(x)<\ell \}$ is convex or empty, for all $\ell\in\R$. 
\end{definition}

The following is perhaps the shortest, albeit naive and incomplete, formulation of a convex sensing problem.  A collection of $n$ points $X=\{x_a\}_{a=1}^n$ in an open convex region $K\subset \mathbb R^d$ is sensed by measuring the values of $m$ {\it sensors}, i.e.  quasi-convex functions $\mathcal{F}=\{f_i:K\to\R\}_{i=1}^m$ . Suppose that one has  access \emph{only} to  the $m\times n$ data matrix $M = [M_{ia}]$ of sensor values, where
\begin{equation}\label{eq:sampling}  M_{ia} = f_i(x_a),\end{equation} 
 but does not have direct access to the  information about the dimension $d$ of the underlying space, the open convex region $K$, the points $x_a\in K$, or any further details of  the quasi-convex functions $f_i$. 
Can one recover any geometric information about the sampled region $K$? At the very minimum, can one infer the dimension $d$?

%This type of convex sensing problem is not unusual in science. We start with a motivation for the problem from neuroscience. 

\subsection{Motivation from neuroscience}\label{subsection:neuro} 
While the convex sensing problems may be not uncommon in many scientific applications,  our chief motivation comes from neuroscience.
Neurons in the brain regions that represent sensory information often possess \emph{receptive fields}.   A paradigmatic example of a receptive field is that of a hippocampal {\it place cell}  \cite{OKeefe1971}. Place cells are  a class of neurons in rodent hippocampus that act as position sensors.   Here the  relevant stimulus space $K\subset \mathbb{R}^d$ is the animal's physical   environment  \cite{Yartsev367}, with $d\in \{1,2,3\}$, and $x\in K $ is the animal's location in this space.  Each neuron is activated with a certain probability that is a continuous function $f\colon K\to \mathbb R_{\geq0}$  of  the animal's position in space. In other words, the probability of a single neuron's activation at a time $t$  is given by  $p(t)=f(x(t))$, where $x(t)$ is the animal's position.  For each neuron, the function $f$ is called its {\it place field}, and is approximately quasi-concave\footnote{A function $f(x)$ is quasi-concave if its negative, $-f(x)$, is quasi-convex.} (see examples of place fields in Figure \ref{fig:PlaceCells}). Place fields can be easily computed when both the neuronal activity data and the relevant stimulus space are available.  A number of  other classes of sensory neurons in the brain also possess quasi-concave {\it receptive fields}, that is, each such neuron  responds with respect to a quasi-concave probability density function $f\colon K\to \mathbb R_{\geq0}$ on the stimulus space. \\

\begin{figure}
  \centering
  \includegraphics[width=0.60\linewidth]{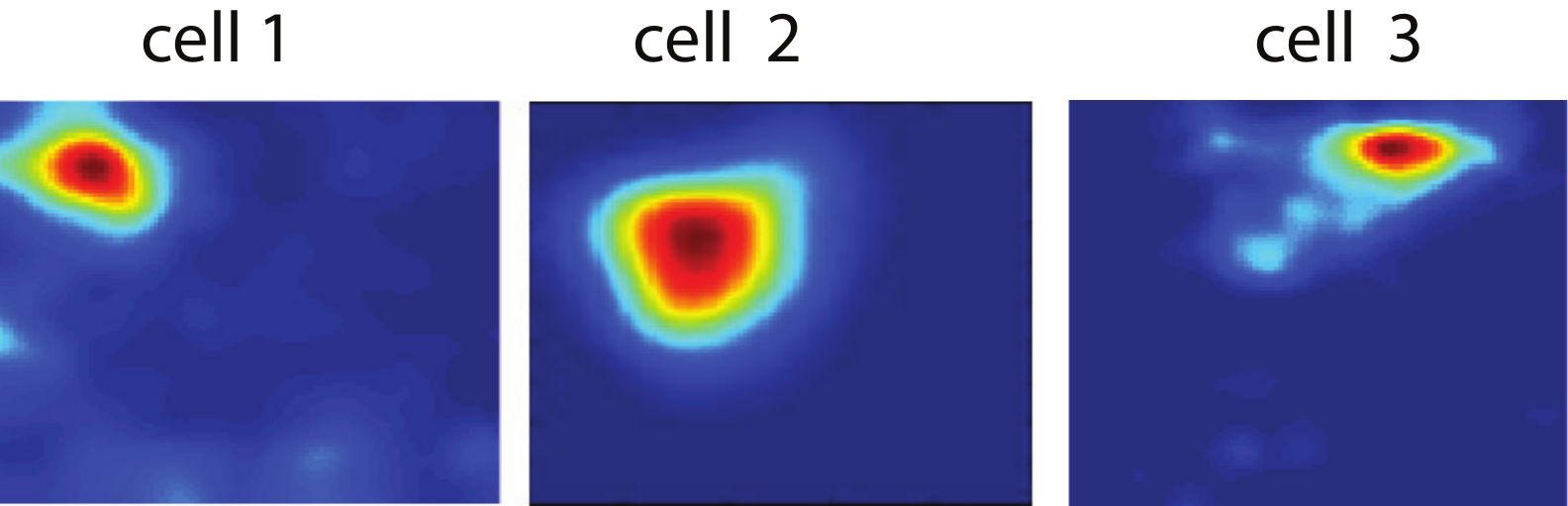}
  \caption{The activities of three different experimentally recorded place cells in a rat's hippocampus. The color represents the probability of each neuron's firing as a function of the animal's location.}	
  \label{fig:PlaceCells}
\end{figure} 

In many situations, the relevant stimulus space for  a given neural population may be unknown.
This raises a natural question: can one infer the dimension of the stimulus space with quasi-concave  receptive fields from neural activity alone? More precisely, given the neural activity of $m$ neurons with quasi-concave receptive fields $f_i: K \to \mathbb R  $,  can one  ``sense'' the stimulus space by  sampling the neural activity at $n$ moments of time as  $M_{ia} =f_i(x(t_a))$?  Here one has access to the measurements $M_{ia}$, but not the objects on the right-hand-side.  This motivates the naive formulation of the convex sensing problem above.

\subsection{The geometry of convex sensing data}
The convex sensing problem possesses a natural transformation group.    If  $\phi \colon \mathbb R \to \mathbb R $ is a strictly monotone-increasing  function,\footnote{A function $\phi:\mathbb R \to \mathbb R$ is strictly monotone-increasing if $\phi(y)>\phi(x)$ whenever $y>x$.}  then the sub-level sets of  the composition $\phi   \circ f $ and $f $ are identical up to an order-preserving re-labeling. 
Thus, if $\phi$ is a strictly monotone-increasing function, then $f$ is quasi-convex if and only if  $\phi \circ f $ is quasi-convex. 
It is easy to show that two sets of real numbers have the same ordering, that is,   $a_1<a_2<\cdots< a_n$ and $b_1<b_2<\cdots< b_n$ if and only if there exists a strictly monotone function   $\phi \colon \mathbb R \to \mathbb R $, such that $b_i=\phi(a_i)$ for all $i$.  It thus follows that it is only the total order of each row in the matrix $M$ in equation \eqref{eq:sampling}  that constrain the geometric features of the point cloud $X_n = \{x_1, ..., x_n\}$ in a convex sensing problem. This motivates the following definition. \\

\begin{definition}
Let $V$ be a finite set.   A {\bf sequence} of length $k$ in $V$ is a $k$-tuple $s = (v_1, ..., v_k)$ of elements in $V$ {\it without repetitions}. We  denote by $S_k[V]$ the set of all sequences of length $k$ on $V$. 
\end{definition}
\medskip

If $M$ is an $m\times n$ real matrix  that has distinct entries in  each row, then  each row yields a sequence of length $n$.  For the sake of an example, consider a real-valued matrix 
\[
M = \begin{bmatrix}
8.23 & 4.19 & 2.56 & 3.96\\
4.78 & 2.88 & 5.76 & 13.43
\end{bmatrix}.
\]
Since the first row has the ordering $2.56<3.96<4.19<8.23$, the total order $<_1$ on $ V= \{1, 2, 3, 4\}$   is  
$ 
3<_1 4<_1 2<_1 1 
 $.
Thus, the order sequence for the first row is $s_1 = (3,4,2,1)\in S_4[V]$. Similarly, the order sequence for the second row is $s_2 = (2,1,3,4)$. \\

It is easy to see that if  the sampled  points $X_n$ and the quasi-convex functions $\{f_i\}_{i\in [m]}$ are  generic  in some natural sense\footnote{It will be rigorously defined in Section \ref{sec:generic:problem}.}, then  each row of the data matrix $M_{ia}  = f_i(x_a)$ has no repeated values with probability $1$. We denote the set  of  all ``generic'' data matrices as
\[
\mathcal{M}_{m,n}^o\od \{m\times n\text{ real-valued matrices with no repeated entries in each row}\}.
\] 
For  any such   matrix $M = [M_{ia}]\in\mathcal{M}_{m,n}^o$, one can define a collection $S(M)$ of $m$ maximal-length sequences as  $S(M)=\{s_1,\dots, s_m\}$, where each sequence 
$$ s_i = (a_{i1}, ..., a_{in})\in{ S_n[n] }\footnote{\text{The accurate notation is $S_n[[n]]$, but here we use the less cumbersome notation $S_n[n]$.}} $$ is obtained from the total order of the $i$-th row: 
 $$M_{ia_{i1}}< M_{ia_{i2}}<...< M_{ia_{in}}.$$

\bigskip 

The geometry of a convex sensing problem for a data matrix $M\in \mathcal{M}_{m,n}^o$  is  constrained only by the set of $m$ sequences $S(M) \subset { S_n[n] } $. The following observation makes it possible to re-state any convex sensing problem purely in terms of embedding a set of points that satisfy certain convex   hull non-containment conditions.   Let  $\conv(x_1, ..., x_k)$  denote the convex hull of a collection of points $x_1, ..., x_k$ in $\R^d$. 

\begin{lemma}\label{lem:convex_hull_noncontainment}
For  %any   sequence $s = (s(1), ..., s(n))$  on $[n]$ of length $n$ and 
any collection of $n$ distinct points
$  \{x_1,x_2,\dots, x_n\}\subset \R^d$, the following statements are equivalent:
\begin{itemize}
\item[(i)] There exists a continuous quasi-convex function $f:\R^d\to\R $     such that  
\begin{equation}\label{eq:ordering-function}  f (x_{1})<f (x_{2})<\cdots<f (x_{ n}), \end{equation} 
\item[(ii)] For each   $k = 2, ..., n$, \quad  $x_{ k }\notin\conv(x_{ 1 }, ..., x_{ k-1 })$.
\end{itemize}
\end{lemma}
\begin{proof}
The implication  $(i)\!\!\implies\!\! (ii)$ follows from Definition \ref{def:quasi-convex}. To prove that $(ii)\!\!\implies\!\! (i)$,  denote $C_k = \conv(x_{ 1 }, ..., x_{ k })$, $d_k(x)\od\dist(x,C_k)$ for any $k=1,\dots n$, and define 
$f (x) = \sum_{k=1}^n h_k\cdot d_k(x)$,
 where $h_1 =  1$ and 
\[
h_k\od 1+
\frac{1}{d_k(x_{ k+1 })} \max\left\{
\sum_{j = 1}^{k-1} h_j\left(d_j\left(x_{ k }\right)-d_j\left(x_{ k+1 }\right)\right), 0
\right\}  , 
 \text{ for } k\geq 2. 
\]
Note that (ii) implies that $d_k(x_{ k+1 })>0$ for all $k\geq 2$. 
Recall that, for any convex set $C\subset \R^d$, the function $x\mapsto \dist(x,C) 
$  is continuous and convex\footnote{See, e.g., Example 3.16 in \cite{Boyd04convexoptimization}.}. Thus,  since  $h_k $ are positive,    $f(x)$ is a continuous convex\footnote{See, e.g., Section 3.2.1 in \cite{Boyd04convexoptimization}.} (and thus quasi-convex)  function. Moreover, $ f(x_{ 1 }) =0 < \dist(x_{ 2 },x_{ 1 })= f (x_{ 2 }) $,  and
\[
h_k>\frac{1}{d_k(x_{ k+1 })}    \sum_{j = 1}^{k-1} h_j\left(d_j\left(x_{ k }\right)-d_j\left(x_{ k+1 }\right)\right)   \text{ for } k\geq2.
\]
The last inequality is equivalent to  $f(x_{ k+1 })>f(x_{ k })$.  Thus inequalities \eqref{eq:ordering-function} hold.
\end{proof}
\begin{corollary}\label{lem_convex:hull:conditions} 
A  matrix $M=[M_{ia}]\in \mathcal{M}_{m,n}^o$  can be obtained as  $M_{ia} =f_i(x_a)$ from a collection  of $m$ continuous quasi-convex functions $  f_i:\mathbb R^d \to\R $  and   $n$ points $ x_1, \dots , x_n\in  \mathbb R^d$ if and only if there exist  $n$ points 
  $ x_1, \dots , x_n\in  \mathbb R^d$ such that, for each sequence  $s=(a_1, a_2,\dots , a_n)\in S(M)$ and each   $k = 2, ..., n$,    $x_{ a_k }\notin\conv(x_{ a_1 }, ..., x_{ a_{k-1} })$.
\end{corollary}

\bigskip An  important implication of Corollary \ref{lem_convex:hull:conditions} is that a convex sensing problem without any further constraint always has a two-dimensional solution.\footnote{By choosing $d=2$ in Corollary \ref{lem_convex:hull:conditions}, we obtain a two-dimensional solution.} Recall that a set of points is \emph{convexly independent} if none of these points lies in the convex hull of the others. 

\begin{corollary}\label{lem_cautionary}
For every matrix $M\in\mathcal{M}_{m,n}^o$ and convexly independent points $x_1,x_2,\dots, x_n \in \R^d$,  
there exist $m$ continuous  quasi-convex functions $f_i:\R^d\to\R $ such that $M_{ia} = f_i(x_a)$.
\end{corollary}
 
\bigskip 
Note that the situation where all the sampled points   are convexly independent is  \emph{non-generic}, for large $n$.   If one  explicitly excludes this situation, then the combinatorics of $S(M)$ constrains the minimal possible dimension $d$ of the geometric realization, as  illustrated by the following example. 
\begin{example}\label{ex_dim_obst}
Let  $n>2$,  and $M\in\mathcal{M}_{n-1,n}^o$ be a matrix  obtained as in equation  \eqref{eq:sampling} with continuous quasi-convex functions $f_i$, whose $(n-1)$ sequences $S(M)=\{s_1,s_2,\dots, s_{n-1}\} $ are of the form 
\begin{align} \label{obstruction:example} 
s_i  =   (\cdots, n, i),  \quad \text{for all } i\in [n-1],
\end{align}
where each of the ``$\cdots$" in $s_i$  is  an arbitrary permutation of $[n]\setminus\{n, i\}$.
Assume that  at least one point in $ X_n = \{x_1, ..., x_n\}\subset \mathbb R^d$ is contained in the interior of the convex hull $\conv(X_n)$, then the dimension in which $M$ can be obtained as in Corollary \ref{lem_convex:hull:conditions} is $d= n-2$. The proof is given in Section \ref{subsection:dim:obstruction} of the Appendix.
\end{example}

\subsection{Dimension inference in a convex sensing problem}
\label{sec:generic:problem}
 It is clear from Corollary  \ref{lem_cautionary} and Example \ref{ex_dim_obst}, that the problem of dimension inference  is well-posed \emph{only} in the presence of some genericity assumptions that guarantee convex dependence of the sampled points. Instead of making such an assumption  explicit,   we take a probabilistic perspective, wherein points are drawn from a probability distribution that is \emph{generic} is some natural sense. We assume that there are three   objects (which are unknown)  that underly any ``convex sensing'' data:
\begin{itemize}
\item[(i)] an open convex set $K\subseteq\R^d$, 
\item[(ii)] $m$ quasi-convex continuous  functions $\mathcal{F}=\{f_i:K\to\R\}_{i=1}^m$, and 
\item[(iii)] a probability measure $P_K$ on $K$. 
\end{itemize}
In relation to the neuroscience motivation in Section \ref{subsection:neuro}, $K\subseteq\R^d$ is the stimulus space, each  function  $f_i$ is  the negative of the receptive field of a neuron,  and $P_K$ is the  measure that describes the probability distribution  of the stimuli.  To guarantee that the convex sensing data  are {\it generic}, we  impose the following regularity assumptions.

\begin{definition}\label{defn:regpair} 
A {\bf regular pair} is a pair $(\mathcal{F}, P_K)$ that satisfies  the  conditions (i)-(iii) above as well as the  following two conditions:
\begin{enumerate}
\item[(R1)] The probability measure $P_K$ is equivalent to the Lebesgue measure on $K$.
\item[(R2)] Level sets of all functions in $ \mathcal{F}$  are of  measure zero, i.e.  for every $i\in [m]$ and $\ell\in\R$, $P_K(f_i^{-1}(\ell))=0$.
\end{enumerate}
\end{definition}

\begin{definition} A point cloud $\{x_1, ..., x_n\}\subset K$ is {\bf sampled from} a regular pair $(\mathcal{F}, P_K)$ if it is i.i.d. from $P_K$. A matrix $M=[M_{ia}]\in \mathcal{M}_{m,n}^o$ is {\bf sampled from} a regular pair $(\mathcal{F}, P_K)$, if for all  $i\in [m]$, and $a\in [n]$, 
$M_{ia} =f_i(x_a),  
$ 
where $\{x_1, ..., x_n\}\subset K $ is sampled from $(\mathcal{F}, P_K)$. 
\end{definition}
 
 \bigskip 
The assumption (R1) ensures  that the  domain $K$ is well-sampled, and thus the probability that  the   points $ x_1, ..., x_n$  are convexly independent approaches zero in the limit of large $n$. The assumption (R2) guarantees, with probability $1$,  that the  data matrix $M$   has no repeated values in each row, and thus is  in $ \mathcal{M}_{m,n}^o$.

\bigskip 

\newcommand{\classofRP}{{\mathcal RP}}

In this paper, we develop a method for estimating the  dimension of  convex sensing data. 
Intuitively, such an estimator needs to be \emph{consistent}, i.e.  ``behave well'' in the limit of large data.  
In addition to  the conditions imposed on a regular pair,  other properties of a pair $(\mathcal{F}, P_K)$ may be needed, depending on the context.   It is therefore natural to define a consistent dimension estimator in relation to a particular class of regular pairs.  Since an estimator may rely on different parameters for different regular pairs, we consider a one-parameter family of such estimators, motivating the following definition of consistency. 

\def \Aconst{asymptotically consistent }
\begin{definition} Let $\classofRP$ be a   class of regular pairs. For each regular pair $\RegPair\in \classofRP$ we denote by $d  \RegPair$ the dimension  $d$, where the open convex set $K\subseteq \mathbb R^d$ is embedded.
A one-parameter family of functions  $\hat d(\varepsilon) \colon \mathcal{M}_{m,n}^o\to \mathbb N$ is called  an {\bf \Aconst estimator in $\classofRP$}, if for every regular pair 
$\RegPair \in \classofRP$, there exists $l>0$, such  that for every $\varepsilon \in (0,l)$ and each  sequence of matrices $M_n  \in  \mathcal{M}_{m,n}^o $,   sampled from   $\RegPair$,  
\begin{equation} \lim_{n\to \infty } P\left (\hat d\left (\varepsilon\right)\left( M_n\right) = d  \RegPair  \right )  =1.
\end{equation} 
\end{definition}

\bigskip 
The structure of this paper is as follows. 
In Section 2, we define two multi-dimensional filtrations of simplicial complexes:   the  empirical Dowker complex $\Dow(S(M))$ that can be associated to a data matrix $M$,  and the Dowker complex 
$\DDow\RegPair$, that can be associated to  a regular pair $\RegPair$. Using an interleaving distance between {\it multi-filtered complexes}, we  prove (Theorem \ref{thm_inter_con})  that for a  sequence $\{ M_n\}$ of data matrices, sampled from a regular pair $\RegPair$, $\Dow(S(M_n))\to \DDow\RegPair$  in probability, as $n\to \infty$.\\

In Section 3, we develop tools for estimating the dimension of $\RegPair$ using  persistent homology.  We define a set of maximal persistent lengths associated to $\DDow\RegPair$ and prove (Lemma \ref{lem_lower_bound}) that a lower bound of the dimension of $\RegPair$ can be derived from these persistent lengths. Next we define another set of maximal persistence lengths from $\Dow(S(M_n))$ and prove (Theorem \ref{thm_asym_L_k}) that they converge  to the maximal persistence lengths associated to $\DDow\RegPair$ in probability, in the limit of large sampling of the data. The rest of Section 3 is devoted to two subsampling procedures for different practical situations,  as well as  simulation results that illustrate that the correct dimension can be inferred with these two methods. \\

In Section 4, we introduce  {\it complete} regular pairs  and prove (Theorem \ref{thm_suff_cond_0}) that the lower bound in Lemma \ref{lem_lower_bound} is  equal to the dimension $d  \RegPair$ for complete regular pairs. 
This establishes (Theorem \ref{thm:d_low_estimate_true_d})  that the dimension estimator introduced in Section \ref{subsection:3.3}  is  an asymptotically consistent estimator in the class of  complete regular pairs. In Section 5,  we define an estimator that can be used to test (Theorem \ref{thm_suff_cond_test}) whether the data matrix is sampled from a complete regular pair. The Appendix (Section \ref{section:Appendix}) contains the proofs of the main theorems as well as some technical supporting lemmas.

%%%%%%%%%%%%%%%%%%%%%%%%%%%%%%%%%%%%%%%%%%%%%%%%%
%%%%%%%%%%%    Section 2

\section{Empirical Dowker complex and the interleaving convergence theorem}
In this section, we define the {\it empirical Dowker complex} from the $m$ sequences induced from the rows of the data matrix $M$ and the {\it Dowker complex} from the regular pair $\RegPair$ and prove that the empirical Dowker complex converges to the Dowker complex in probability. These complexes are both examples of {\it multi-filtered simplicial complexes}.
\begin{definition}
Let $I = \prod_{i\in [m]} I_i$ be an  $m$-orthotope  in $\R^m$, where each $I_i$ is an interval (open, closed, half-open, finite, or infinite are all allowed) in $\R$. Let $\leq$ be the natural partial order on $I$ induced from $\R^m$. A {\bf multi-filtered simplicial complex} $\mathcal{D}$ indexed over $I$ is a collection $\{\mathcal{D}_\alpha\}_{\alpha\in I}$ of simplicial complexes on a fixed finite vertex set, such that, $\mathcal{D}_\alpha\subseteq \mathcal{D}_\beta$, for all $\alpha\leq\beta$ in $I$.
\end{definition}

We define the empirical Dowker complexes from a collection of sequences of maximal length\footnote{i.e. of length $n$} on the vertex set $[n]$.
\begin{definition}\label{def_Dowker_emp}
Let $S=\{s_1, ..., s_m\}$ be a collection of sequences on $[n]$ of length $n$. Let $\leq_i$ be the total order on $[n]$ induced from $s_i$; namely, for $a,b\in [n]$, $a\leq_i b$ if and only if $a$ is before or equal to $b$ in $s_i$. We define the following multi-filtered simplicial complex, with vertex set $[m]$ and indexed over $[0,1]^m$:
\[
\Dow(S) \od \left\{\Dow\left(S\right)\left(t_1, ..., t_m\right):\left(t_1, ..., t_m\right)\in [0,1]^m\right\},
\]
where 
\[
\Dow(S)(t_1, ..., t_m) \od \Delta(\{\sigma_a:a=1, ..., n\}),
\]
and
\[
\sigma_a =\left \{i\in [m]: \#(\{b\in [n]:b\leq_i a\})\leq n  t_i\right\}. 
\]
Here $\Delta(\{ \sigma_a\}_{a\in [n]})$ denotes the smallest simplicial complex containing the faces  $\{ \sigma_a\}_{a\in [n]}$. This filtered complex is called the {\bf empirical Dowker complex} of  $S$.
\end{definition}

Recall from Section 1.2 that the relevant geometric information of the $m\times n$ data matrix $M\in\mathcal{M}_{m,n}^o$ is contained  in the collection of $m$ sequences $S(M) = \{s_1, ..., s_m\}$, where $s_i\subset S_n[n]$ is of length $n$ and  records the total order   induced by the $i$-th row of $M$. Therefore, we can consider the empirical Dowker complex $\Dow(S(M))$ derived from the data matrix $M$. \\

Note that our definition of empirical Dowker complex is a multi-parameter generalization of the Dowker complex defined in \cite{Chazal2014}. Specifically, the one-dimensional filtration of simplicial complex (indexed over $t$) $\Dow(S(M))(n\cdot t, ..., n\cdot t)$ is equal to the Dowker complex defined in \cite{Chazal2014}. \\

% deleted some junk -sentence here 
 Recall that, for a collection $\mathcal{A} = \{A_i\}_{i\in [m]}$ of sets, the {\bf nerve} of $\mathcal{A}$, denoted $\nerve(\mathcal{A})$, is the simplicial complex on the vertex set $[m]$ defined as  
\[
\nerve(\mathcal{A}) \od \left\{\sigma\subseteq [m]: \bigcap_{i\in\sigma} A_i\neq\varnothing\right\}.
\]
The following lemma is immediate from Definition \ref{def_Dowker_emp}. 
\begin{lemma}\label{lem_Dow(S)_nerve}
Let $S=\{s_1, ..., s_m\}$ be a collection of sequences on $[n]$ of length $n$. For each $i\in [m]$ and $t\in\R$, consider 
$$A^{(i)}(t) \od \left \{a\in [n]: \#\left (\left \{b\in [n]:b\leq_i a\right \}\right)\leq n t\right\}\subset [n],$$ 
where $\leq_i$ is the total order on $[n]$ induced by $s_i$. Then 
\[
\Dow(S)(t_1, ..., t_m) = \nerve\left (\left \{A^{(i)}(t_i)\right\}_{i\in [m]}\right).
\]
\end{lemma}

Next we connect the combinatorics of $\Dow(S(M))$  to the geometry. From Lemma \ref{lem_Dow(S)_nerve}, we know that $\Dow(S(M))$ is the nerve of $\{A^{(i)}(t_i)\}_{i\in [m]}$. To define an analogue of $\Dow(S(M))$ from   the regular pair $(\mathcal{F},P_K)$, we use the following  lemma (see the proof  in Section \ref{AppSec:ell_i(t)}) to define an analogue of $A^{(i)}(t)$ from $\RegPair$.

\begin{lemma}\label{lem:ell_t}
Let $f:K\to\R$ be a continuous  function with $P_K(f^{-1}(\ell)) = 0$ for all $\ell\in\R$, where $P_K$ is a probability measure on a convex open set $K$ and $P_K$ is equivalent to the Lebesgue measure on $K$. Then there exists a unique strictly increasing continuous function $\lambda:(0,1)\to\R$ such that, for all $t\in (0,1)$, 
\begin{equation}\label{eq:ell_t}
P_K(f^{-1}(-\infty,\lambda(t)) = t.
\end{equation}
\end{lemma}

For a regular pair $\RegPair=(\{f_i:K\to\R\}_{i\in [m]},P_K)$, by Lemma \ref{lem:ell_t}, for each $i\in [m]$, there exists a unique strictly increasing continuous function $\lambda_i:(0,1)\to\R$ such that $P_K(f_i(-\infty,\lambda_i(t)) = t$. Using $\lambda_i(t)$, the following definition provides a continuous analogue of $A^{(i)}(t)$.

\begin{definition}
Let $(\mathcal{F},P_K)=(\{f_i\}_{i\in [m]},P_K)$ be a regular pair. For each $i\in [m]$ and $t\in (0,1)$, define
\[
K^{(i)}(t)\od f_i^{-1}(-\infty,\lambda_i(t)),
\]
where $\lambda_i :(0,1)\to\R$ is the unique function that satisfies $P_K(f_i^{-1}(-\infty,\lambda_i(t))) = t$. For convenience, we also define $K^{(i)}(0)\od\varnothing$ and $K^{(i)}(1)\od K$.
\end{definition}

\begin{figure}[H]
  \centering
  \includegraphics[width=0.40\linewidth]{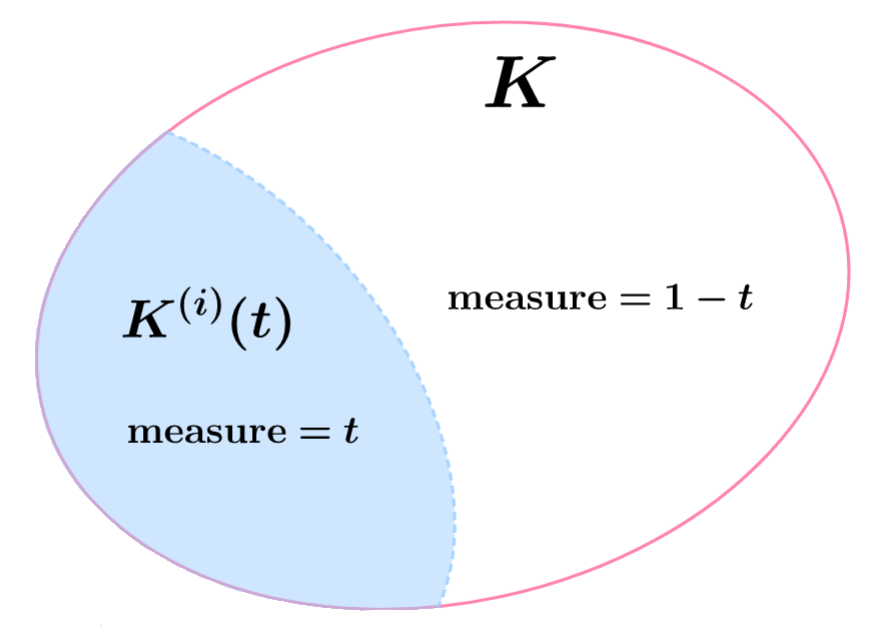}
  \caption{$K^{(i)}(t)$ is the sublevel set of $f_i$ whose $P_K$ measure is equal to $t$.}	
  \label{fig:Kit}
\end{figure}

An illustration of $K^{(i)}(t)$ can be found in Figure \ref{fig:Kit}. They are simply sublevel sets of $f_i$ rescaled with respect to the $P_K$ measure. On the other hand, for a point cloud $X_n = \{x_1,\cdots,x_n\}$ sampled from $P_K$, if we identify $[n]$ with $X_n$ via $a\leftrightarrow x_a$, then $A^{(i)}(t)$ may be interpreted as the set of points in $X_n$ that is inside an approximation of $K^{(i)}(t)$. Informed by  Lemma \ref{lem_Dow(S)_nerve}, we use $K^{(i)}(t)$ to define the continuous version of Dowker complex. 

\begin{definition}\label{def_Dowker}
Let $(\mathcal{F}, P_K)$ be a regular pair. Define a multi-filtered complex $\DDow(\mathcal{F}, P_K)$, indexed over $[0,1]^m$, by 
\[
\DDow(\mathcal{F}, P_K)(t_1, ..., t_m)\od \nerve\left(\{K^{(i)}(t_i)\}_{i=1}^m\right).
\]
This multi-filtered complex is called the {\bf Dowker complex} induced from $(\mathcal{F}, P_K)$.
\end{definition}

The complex $\Dow(S(M))$ is what we can obtain from the data matrix $M$, but it does not capture the whole geometric information of $(\mathcal{F},P_K)$. On the other hand, $\DDow(\mathcal{F}, P_K)$ reflects the whole geometric information but is not directly computable. Since $A^{(i)}(t_i)$ is an approximation of $K^{(i)}(t_i)$, we might expect $\Dow(S(M))$ approximates $\DDow(\mathcal{F}, P_K)$. As we shall see, this is the case but, for comparing them formally, we need the concept of the interleaving distance.
\begin{definition}
For a multi-filtered complex $\mathcal{K}$ indexed over $\R^m$ and $\epsilon>0$, the {\bf $\epsilon$-shift} of $\mathcal{K}$, denoted $\mathcal{K}+\epsilon$, is the multi-filtered complex defined by 
\[
(\mathcal{K}+\epsilon)(t_1, ..., t_m) \od \mathcal{K}(t_1+\epsilon, ..., t_m+\epsilon).
\]
For two multi-filtered complexes $\mathcal{K}$ and $\mathcal{L}$ indexed over $\R^m$, the {\bf simplicial interleaving distance} between $\mathcal{K}$ and $\mathcal{L}$ is defined as 
\[
d_{\INT}(\mathcal{K}, \mathcal{L})
\od \inf\{\epsilon>0:
\mathcal{K}\subseteq \mathcal{L}+\epsilon
\text{ and }
\mathcal{L}\subseteq \mathcal{K}+\epsilon\}.
\]
\end{definition}
Note that this interleaving distance is between multi-filtered simplicial complexes while the standard interleaving distance in topological data analysis is between {\it persistence modules}, namely, the level where the homology functor has been applied on the multi-filtered complex (see, e.g. \cite{Lesnick_2015}, for the standard definition of interleaving distance between multi-dimensional persistence modules). Similar to the standard interleaving distance, the simplicial interleaving distance $d_\INT$ defined here is also a pseudo-metric; namely, $d_\INT(\mathcal{K},\mathcal{L}) = 0$ does not imply $\mathcal{K}=\mathcal{L}$.

The definition of simplicial interleaving distance involves a shift of indices and that is why the two multi-filtered complexes to be compared are required to be indexed over the whole $\R^m$. Since both $\Dow(S(M))$ and $\DDow(\mathcal{F}, P_K)$ are indexed only over $[0,1]^m$, to compare them in terms of interleaving distance, we first need to extend their indexing domain to $\R^m$. The definition below is a natural way to extend the indexing domain.
\begin{definition}
For $\mathcal{D}=\Dow(S(M))$ or $\DDow(\mathcal{F}, P_K)$ and $(t_1, ..., t_m)\in\R^m$, define 
\[
\mathcal{D}(t_1, ..., t_m) \od \mathcal{D}(\theta(t_1),\cdots,\theta(t_m))
\]
where $\theta:\R\to [0,1]$ is defined by $\theta(t) = t$, if $0\leq t\leq 1$; $\theta(t) =0$, if $t<0$; $\theta(t) =1$, if $t>1$.
\end{definition}

With the above notations, we state one of our main theorems.
\begin{theorem}[Interleaving Convergence Theorem]\label{thm_inter_con}
Let $(\mathcal{F}, P_K)$ be a regular pair and $M_n$ be an $m\times n$ data matrix sampled from $\RegPair$. Then the simplicial interleaving distance between $\Dow(S(M_n))$ and $\DDow(\mathcal{F}, P_K)$ converges to $0$ in probability as $n\to\infty$; namely, for all $\epsilon>0$, 
\[
\lim_{n\to\infty} \Pr\left[d_{\INT}(\Dow(S(M_n)), \DDow(\mathcal{F}, P_K))>\epsilon\right]=0.
\]
\end{theorem}
The proof of Theorem \ref{thm_inter_con} is given in Section 6.2. In Section \ref{sec:app_to_dim_infer_prob}, we use  Theorem \ref{thm_inter_con} to infer a lower bound for the dimension of $\RegPair$.

%%%%%%%%%%%%%%%%%%%%%%%%%%%%%%%%%%%%%%%%%%%%%%%%%%%%%%%%%%%%%%%
%%%%%%%%%%% Section 3

\section{Estimating the stimulus space dimension}
\label{sec:app_to_dim_infer_prob}
\subsection{Persistence modules and maximal persistence length}
First we recall the definition of persistence modules, persistence intervals and persistence diagrams; for more details see, e.g. Chapter 1 of \cite{Oudot2015PersistenceT}. Then we define the maximal persistence length for a 1-dimensional filtration of simplicial complexes. We fix a ground field $\F$, which is normally taken to be $\F_2$ for computational reasons; all the statements here do not depend on the choice of the  field.
\begin{definition}
A {\bf persistence module} $\mathcal{M}$ indexed over an interval $[0,T]$ is a collection $\{\mathcal{M}_t\}_{t\in [0,T]}$ of vector spaces over $\F$ along with linear maps $\phi_s^t:\mathcal{M}_s\to\mathcal{M}_t$ for every $s\leq t$ in $[0,T]$ such that $\phi_s^u = \phi_t^u\circ\phi_s^t$, and     $\phi_t^t =\operatorname{id_{\mathcal{M}_t}}$ for all $s\leq t\leq u$ in $[0,T]$.
\end{definition}

A well-known structural characterization of a persistence module is via its {\it persistence intervals} (or equivalently, its {\it persistence diagram}). To talk about persistence intervals, we would need to define the {\it direct sum} of persistence modules and {\it interval modules}.

\begin{definition}
Let $\mathcal{M} = \{\mathcal{M}_t\}_{t\in [0,T]}$ and $\mathcal{N} = \{\mathcal{N}_t\}_{t\in [0,T]}$ be persistence modules over the same index interval $[0,T]$. Let $\{\phi_s^t:s, t\in [0,T], s\leq t\}$ and $\{\psi_s^t:s, t\in [0,T], s\leq t\}$ be the linear maps of $\mathcal{M}$ and $\mathcal{N}$. The {\bf direct sum} of $\mathcal{M}$ and $\mathcal{N}$, denoted $\mathcal{M}\oplus\mathcal{N}$, is the persistence module, defined by $(\mathcal{M}\oplus\mathcal{N})_t\od\mathcal{M}_t\oplus\mathcal{N}_t$ along with the linear maps $(\phi_s^t)\oplus (\psi_s^t):\mathcal{M}_s\oplus\mathcal{N}_s\to\mathcal{M}_t\oplus\mathcal{N}_t$ for every $s\leq t$ in $[0,T]$.
\end{definition}

\begin{definition}
Let $J\subseteq [0,T]$ be an interval in $[0,T]$, which can be either open, closed, or half-open. The {\bf interval module} $\I_J$ defined over $[0,T]$ is the persistence module $\I_J = \{(\I_J)_t\}_{t\in [0,T]}$ defined by $(\I_J)_t\od\F$ for all $t\in J$ and $(\I_J)_t\od 0$ for all $t\notin J$, along with the identity linear maps from $(\I_J)_s$ to $(\I_J)_t$ for every $s\leq t$ in $J$ and zero maps from $(\I_J)_s$ to $(\I_J)_t$ for other $s\leq t$ in $[0,T]$.
\end{definition}

The next decomposition theorem is a structural theorem that characterizes persistence modules and guarantees the existence and uniqueness of persistence intervals (see, e.g., Section 1.1 and 1.2 of \cite{Oudot2015PersistenceT} and references therein). 
\begin{theorem}\label{thm:decomp_per_homology}
Let $\mathcal{M} = \{\mathcal{K}_t\}_{t\in [0,T]}$ be a persistence module over a closed interval $[0,T]$. If, for each $t\in [0,T]$, $\mathcal{M}_t$ is a finite dimensional vector space over $\F$, then $\mathcal{M}$ can be decomposed as a direct sum of interval modules; namely, 
\begin{equation*}
\mathcal{M} = \bigoplus_{J} \I_J
\end{equation*}
where $\{J\}$ is a collection of some intervals (could be open, closed, or half-open) in $[0,T]$. The decomposition is unique in the sense that, for every such decomposition, the collection of intervals is the same.
\end{theorem}

Each interval $J$ in the decomposition stated in Theorem \ref{thm:decomp_per_homology} is called a {\bf persistence interval} of $\mathcal{M}$. We may summarize all persistence intervals as a 2D diagram in $[0,T]\times [0,T]$, called the {\it persistence diagram} of $\mathcal{M}$: for each persistence interval with left end $\alpha$ and right end $\beta$, we mark a point $(\alpha, \beta)$ in $[0,T]\times [0,T]$. The diagram consisting of all such points is called the {\bf persistence diagram} of $\mathcal{M}$, denoted $\dgm(\mathcal{M})$. Rigorously speaking, we should distinguish open, closed, and half-open intervals. For our purpose, we only use the lengths of the persistence intervals, and hence the distinction of open, closed, and half-open intervals does not really matter. 

An important class of persistence modules is obtained from a 1-dimensional filtration of simplicial complexes by applying the homology functors $H_k(\ \cdot\ ;\F)$, $k = 0, 1, 2, ...$. Specifically, for a 1-dimensional filtration of simplicial complexes $\mathcal{K} = \{\mathcal{K}_t\}_{t\in [0,T]}$ and a fixed nonnegative integer $k$, we have the persistence module $H_k(\mathcal{K};\F) = \{H_k(\mathcal{K}_t;\F)\}_{t\in [0,T]}$ along with the linear maps $(i_s^t)_*:H_k(\mathcal{K}_s;\F)\to H_k(\mathcal{K}_t;\F)$ for every $s\leq t$ in $[0,T]$, where $i_s^t$ is the inclusion map from $\mathcal{K}_s$ to $\mathcal{K}_t$. Since $H_k(\ \cdot\ ;\F)$ is a covariant functor, the equality $(i_t^u)_*\circ (i_s^t)_* = (i_s^u)_*$ holds for every $s\leq t\leq u$ in $[0,T]$. 

For each $k$, we may use the persistence diagram of $H_k(\mathcal{K};\F)$ for analysis. For our purpose, instead of the whole diagram, we summarize the diagram by only looking at the longest length among all persistence intervals, which we formally define below:
\begin{definition}
Let $\mathcal{K} = \{\mathcal{K}_t\}_{t\in [0,T]}$ be a 1-dimensional filtration of simplicial complexes. For each nonnegative integer $k$, we define 
\begin{equation}
l_{\max}(k,\mathcal{K})\od\sup\{\beta-\alpha:(\alpha,\beta)\in H_k(\mathcal{K};\F)\}
\end{equation}
and call it the {\bf maximal persistence length in dimension {\boldmath$k$}}. 
\end{definition}
This definition is similar to the one used in Section 3 of \cite{bobrowski2015maximally}.\footnote{The only difference is that the authors in \cite{bobrowski2015maximally} measures the maximal cycle multiplicatively while we measure it additively.}  Normally, the length of a persistence interval in $H_k(\mathcal{K};\F)$ is viewed as its significance in dimension $k$. Therefore, $l_{\max}(k,\mathcal{K})$, the maximum among such interval lengths, is viewed as the significance of $\mathcal{K}$ in dimension $k$. 

\subsection{$L_k\RegPair$ and its relation to the dimension of $\RegPair$}
In this section, from the regular pair $\RegPair$, we define quantities that we use to bound the dimension $d\RegPair$ from below. We start with the following notation.  
\begin{definition}
Given $(\mathcal{F}, P_K)$, where $\mathcal{F} = \{f_i:K\to\R\}_{i\in [m]}$ is a collection of quasi-convex functions defined on a convex open set $K$ and $P_K$ is a probability measure on $K$. For $x\in K$, we define 
\[
T_i(x) \od P_K(f_i^{-1}(-\infty,f_i(x))).
\]  
\end{definition}
\begin{figure}[H]
  \centering
  \includegraphics[width=0.35\linewidth]{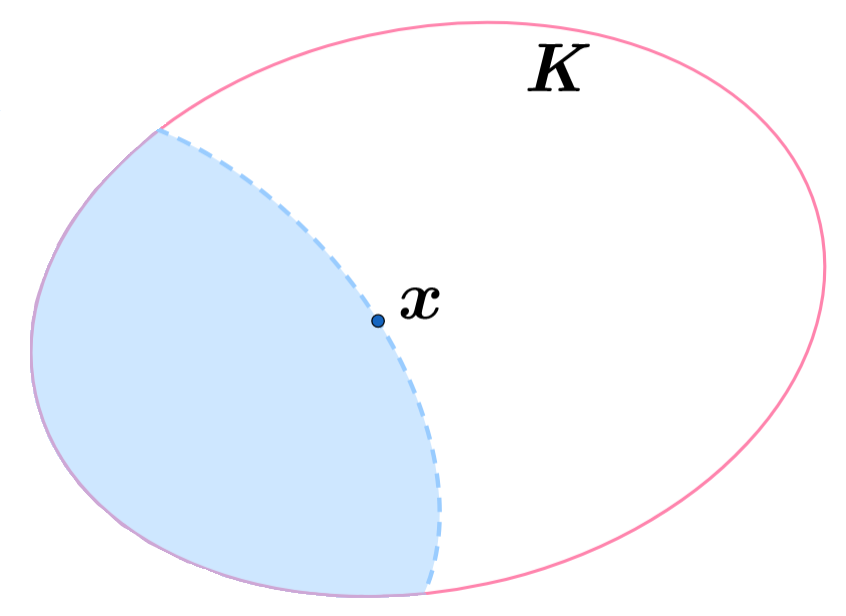}
  \caption{$T_i(x)$ is the $P_K$ measure of the shaded area $f_i^{-1}(-\infty,f_i(x))$.}	
  \label{fig:Tix}
\end{figure}
 The function  $T_i(x)$ may be regarded as the {\it $P_K$-rescaled version of $f_i$} (see Figure \ref{fig:Tix} for an illustration). Now we   define a one dimensional filtration of simplicial complexes $\mathcal{K}_x$   that are used to infer a lower bound of the dimension $d\RegPair$. The geometry  underlying the definition is depicted in Figure \ref{fig:dim_infer_geo_idea}.
 
 % Then we use $\mathcal{K}_x$ to define, for each $k = 0, 1, ...$, $L_k\RegPair$, which are related to the estimator that we use to infer the lower bound. Here are the formal definitions. 

\begin{figure}[H] 
  \centering
  \includegraphics[width=0.3\linewidth]{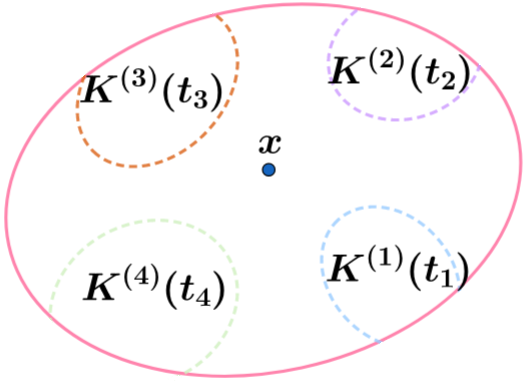}
  \includegraphics[width=0.3\linewidth]{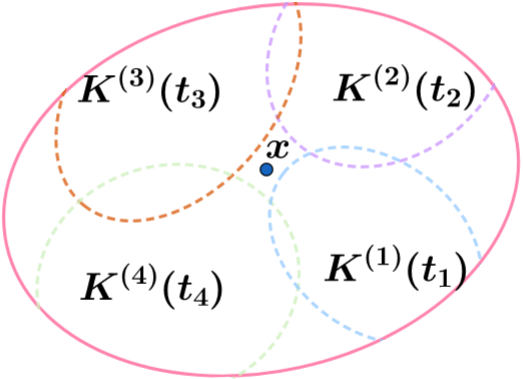}
  \includegraphics[width=0.3\linewidth]{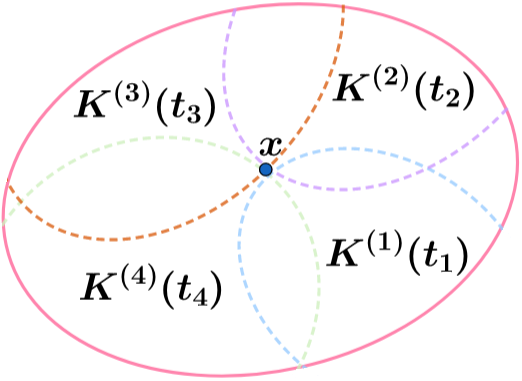}
  \caption{From left to right, the filtration $\mathcal{K}_x(t)$, as the nerve of the sublevel sets of $\{f_i\}_{i\in [m]}$, starts at $t=0$ as the empty simplicial complex and increases as $t$ goes up to $t_{\max}(x)$, where the sublevel sets of $\{f_i\}_{i\in [m]}$ touch the point $x$ on their boundaries. The formal formulation of the process is in \eqref{eq:L_k_eq1} of Definition \ref{def_L_k}.}
  \label{fig:dim_infer_geo_idea}
\end{figure}

Throughout Section \ref{sec:app_to_dim_infer_prob}, we fix an arbitrary coefficient field $\F$ when taking homology; namely, for a filtration of simplicial complexes $\mathcal{K}$ and a nonnegative integer $k$, $H_k(\mathcal{K})\od H_k(\mathcal{K};\F)$.

\begin{definition}\label{def_L_k}
Let $\RegPair$ be a regular pair, where $\mathcal{F} = \{f_i:K\to\R\}_{i\in [m]}$. For $x\in K$, let 
\begin{equation}
t_{\max}(x) \od \max_{i\in [m]} T_i(x).
\end{equation}
Define a one dimensional filtered complex $\mathcal{K}_x$, indexed over $t\in [0,t_{\max}(x)]$, by 
\begin{equation}\label{eq:L_k_eq1}
\mathcal{K}_x(t) \od \DDow(\mathcal{F},P_K)\left(T_1(x)-(t_{\max}(x)-t),...,T_m(x)-(t_{\max}(x)-t)\right).
\end{equation}
For every nonnegative integer $k$, we define
\begin{equation}\label{eq:L_k_eq2}
L_k(\mathcal{F},P_K) \od \sup_{x\in K} l_{\max}(k,\mathcal{K}_x).
\end{equation}
\end{definition}

As illustrated in Figure \ref{fig:dim_infer_geo_idea}, if $x$ is ``central" in some appropriate sense (see Definition  \ref{def:Cent_1}  in Section \ref{sec:suff_cond}), a $(d\RegPair-1)$-dimensional sphere is expected to show up and persist for a significant amount of time. In general, $L_k\RegPair$ can at least be used to derive a lower bound for the dimension of the regular pair $(\mathcal{F}, P_K)$ due to the following lemma.
\begin{lemma}\label{lem_lower_bound} 
Let $\RegPair$ be a regular pair. Then, for $k\geq d\RegPair$, $L_k(\mathcal{F},P_K)=0$. In particular, 
\begin{equation}\label{eq:d_low:definition}  
d_{\low}\RegPair\od 1+\max\{k:L_k(\mathcal{F},P_K)> 0\} \leq d\RegPair.		
\end{equation}
\end{lemma}

\begin{proof}
For notational simplicity, in this proof, we denote $d_{\low}=d_{\low}\RegPair$ and $d = d\RegPair$. Recall that $\DDow(\mathcal{F},P_K)(t_1,..., t_m) = \nerve\left(\{f_i^{-1}(-\infty,\lambda_i(t_i))\}_{i=1}^m\right)$.
Since the  functions $f_i$ are quasi-convex, intersections of convex sets are convex and convex sets are contractible, the set $\{f_i^{-1}(-\infty,\lambda_i(t_i))\}_{i=1}^m$ is a good cover. Thus, by nerve lemma (see, e.g., Theorem 10.7 in \cite{Bjorner:1996:TM:233228.233246} or Corollary 4G.3 in \cite{MR1867354}), we have the following  homotopy equivalence:
\begin{equation}\label{eq:nerve_homotopy}
\DDow(\mathcal{F},P_K)(t_1,..., t_m)\sim\bigcup_{i\in [m]} f_i^{-1}(-\infty,\lambda_i(t_i)).
\end{equation}
Notice that $\bigcup_{i\in [m]} f_i^{-1}(-\infty,\lambda_i(t_i))$ is open in $\R^d$ and it is well-known that, for every open set $U\subseteq\R^d$, $H_k(U) = 0$, for all $k\geq d$ (see, e.g., Proposition 3.29 in \cite{MR1867354}). Thus, for $k\geq d$, $H_k\left(\bigcup_{i\in [m]} f_i^{-1}(-\infty,\lambda_i(t_i))\right) = 0$. Combining with \eqref{eq:nerve_homotopy}, we obtain, for $k\geq d$ and $(t_1, ..., t_m)\in [0,1]^m$, $H_k(\DDow(\mathcal{F},P_K)(t_1,..., t_m)) = 0$. Therefore, for $k\geq d$, $l_{\max}(k,\mathcal{K}_x)=0$, for all $x\in K$, and $L_k\RegPair=0$. Thus $d_{\low}-1 = \max\{k:L_k(\mathcal{F},P_K)> 0\}\leq d-1$ or, equivalently, $d_{\low}\leq d$. 
\end{proof}
$L_k(\mathcal{F},P_K)$ is defined with respect to a regular pair $(\mathcal{F}, P_K)$ and thus is not directly computable from discrete data. In Section \ref{subsection:3.3}, we follow an analogous approach in defining $L_k\RegPair$ to define $L_k(M)$ and prove that $L_k(M)$ converges to $L_k\RegPair$. 

\subsection{$L_k(M)$ and its convergence to $L_k\RegPair$}\label{subsection:3.3}  
In Theorem \ref{thm_inter_con}, we see that, for the data matrix $M$, $\Dow(S(M))$ approximates $\DDow(\mathcal{F}, P_K)$ with high probability. Thus, it is natural to use $\Dow(S(M))$ to define an analogue $L_k(M)$ of $L_k\RegPair$.

\begin{definition}\label{def_L_k_emp}
Let $M\in\mathcal{M}_{m,n}^o$ and $S(M)=\{s_1, ..., s_m\}$ be the collection of $m$ sequences induced from the rows of $M$ ($s_i$ corresponds to row $i$). For $a\in [n]$ and $i\in [m]$, denote
\begin{equation}
\ord_i(M,a)\od \#(\{b\in [n]:M_{ib}\leq M_{ia}\}).
\end{equation}
For $a\in [n]$, which corresponds to the $a$-th column of the data matrix $M$, let 
\begin{equation}
\hat{t}_{\max}(a)\od \max_{i\in [m]} \frac{\ord_i(M,a)}{n}.
\end{equation}
Define a one dimensional filtered complex $\hat{\mathcal{K}}_{a}$, indexed over $t\in [0,\hat{t}_{\max}(a)]$, by
\begin{equation}
\hat{\mathcal{K}}_a(t) \od \Dow(S(M))\left(\frac{\ord_1(M,a)}{n}-(\hat{t}_{\max}(a)-t), ..., \frac{\ord_m(M,a)}{n}-(\hat{t}_{\max}(a)-t)\right).
\end{equation}
See Definition \ref{def_Dowker_emp} for the definition of $\Dow(S(M))$. For every nonneative integer $k$, we define 
\begin{equation}\label{eq:L_k_emp_eq2}
L_k(M) \od \max_{a\in [n]} l_{\max}(k,\hat{\mathcal{K}}_a).
\end{equation}
\end{definition}

Since $\Dow(S(M))$ approximates $\DDow(\mathcal{F},P_K)$, intuitively, $\hat{\mathcal{K}}_a(t)$ approximates $\mathcal{K}_{x_a}(t)$ and $L_k(M)$ approximates $L_k(\mathcal{F},P_K)$. With the help of  Theorem \ref{thm_inter_con} and the Isometry Theorem in topological data analysis  (see e.g. Theorem \ref{thm:isometry_thm} in Section \ref{section:Appendix} of \cite{Oudot2015PersistenceT}), these intuitions are justified as follows:

\begin{theorem}\label{thm_asym_L_k}
Let $(\mathcal{F}, P_K)$ be a regular pair. Assume that $K$ is bounded and each $f_i\in\mathcal{F}$ can be continuously extended to the closure $\bar{K}$. Let $M_n$ be an $m\times n$ matrix sampled from $\RegPair$. Then, for all $k\in\{0\}\cup\N$, as $n\to\infty$, $L_k(M_n)$ converges to $L_k(\mathcal{F},P_K)$ in probability; namely, for all $\epsilon>0$, 
\[
\lim_{n\to\infty} \Pr\left[\left|L_k(M_n)-L_k(\mathcal{F},P_K)\right|<\epsilon\right]=1.
\]
Moreover, the rate of convergence is independent of $k$.\footnote{i.e. for all $\epsilon>0$, $\lim_{n\to\infty} \Pr\left[\sup_{k\geq 0}\left|L_k(M_n)-L_k(\mathcal{F},P_K)\right|<\epsilon\right]=1$}
\end{theorem}

The proof of Theorem \ref{thm_asym_L_k} is given in Section \ref{sec:pf_of_asym_L_k}. According to Theorem \ref{thm_asym_L_k}, for each non-negative integer $k$, $L_k(M_n)$ are consistent estimators of $L_k(\mathcal{F},P_K)$ and they converge uniformly in probability. Thus, by Lemma \ref{lem_lower_bound} and Theorem \ref{thm_asym_L_k}, we can estimate a lower bound for the dimension of $\RegPair$ from the data matrix $M$, via looking at the values of $L_k(M_n)$. Formally, we can define the following estimator of $d_\low\RegPair$.
\begin{definition}\label{def:est_low}
For  $\epsilon>0$  and  $M\in\mathcal{M}_{m,n}^o$, we define
\begin{equation}
\hat{d}_{\low}(M,\epsilon)\od 1 + \max\{k:L_k(M)>\epsilon\}.
\end{equation}
\end{definition}
As a consequence of  Lemma \ref{lem_lower_bound} and Theorem \ref{thm_asym_L_k}, it is immediate that $\hat{d}_{\low}(M_n,\epsilon)$ is a consistent estimator, for appropriately chosen $\epsilon$.
\begin{corollary}\label{cor:d_low_estimator}
Let $\RegPair$ be a regular pair satisfying the conditions in Theorem \ref{thm_asym_L_k} and $M_n\in\mathcal{M}_{m,n}^o$ be sampled from $\RegPair$. Denote $d_\low = d_\low\RegPair$. Then, for all $0<\epsilon<L_{d_{\low}-1}\RegPair$, 
\begin{equation}
\lim_{n\to\infty} \Pr\left[\hat{d}_{\low}(M_n,\epsilon) = d_{\low}\RegPair\right] = 1.
\end{equation}
\end{corollary}
\begin{proof}
For notational simplicity, in this proof, we denote $d = d\RegPair$. By Lemma \ref{lem_lower_bound} and Theorem \ref{thm_asym_L_k}, for $k\geq d$, $L_k(M_n)\to 0$ in probability and $L_{d_\low-1}(M_n)\to L_{d_\low-1}\RegPair>0$ in probability, as $n\to\infty$, with the same rate of convergence. Since $0<\epsilon<L_{d_{\low}-1}\RegPair$, as $n\to\infty$, w.h.p., $L_{d_\low-1}(M_n)>\epsilon$ and $L_k(M_n)<\epsilon$. Therefore, w.h.p., $\hat{d}_{\low}(M_n,\epsilon) = d_{\low}$, and the result follows.
\end{proof}

From Corollary \ref{cor:d_low_estimator}, $\hat{d}_{\low}(M_n,\epsilon)$ can be used as a consistent estimator of $d_\low\RegPair$. However, we need to know how to choose an appropriate $\epsilon$ for $\hat{d}_{\low}(M_n,\epsilon)$, and hence estimation of $L_k\RegPair$ is still necessary. Therefore, in practice, we suggest one use a statistical approach estimating $L_k\RegPair$ to infer $d_\low\RegPair$, instead of using $\hat{d}_{\low}(M_n,\epsilon)$ directly. The details are discussed in Section \ref{sec:use_L_k_stat_subsampling}.

\subsection{Algorithm for $\hat{\mathcal{K}}_a$ and $L_k(M)$}
For ease of implementation, we combine Definition \ref{def_L_k_emp} and Definition \ref{def_Dowker_emp} and summarize them as algorithms for the computaion of $\hat{\mathcal{K}}_a$ and $L_k(M)$. Algorithm \ref{alg:RepFiltration} is for $\hat{\mathcal{K}}_a$.
\begin{algorithm}[H]
\caption{Computation of $\hat{\mathcal{K}}_a$}\label{alg:RepFiltration}
\begin{algorithmic}
\State{}
\State{{\bf INPUTS:}}
\State{(1) $M$: an $m\times n$ real matrix without repeated values on each row; namely, a matrix in $\mathcal{M}^o_{m,n}$.}
\State{(2) $a$: an integer in $[n]$, referring to the $a$-th column of $M$.}
\State{}
\State{{\bf OUTPUT:}}
\State{$[\hat{\mathcal{K}}_a(t)]_t$: a filtration of simplicial complexes, where $t\in\{0, 1/n, 2/n, ..., \hat{t}_{\max}(a)\}$ (see Step 2 or Denfinition \ref{def_L_k_emp} for definition of $\hat{t}_{\max}(a)$).}
\State{}
\State{{\bf STEPS:}}
\State{{\bf Step 1:} For $i\in [m]$ and $b\in [n]$, recall from Definition \ref{def_L_k_emp} the order (a positive integer) 
\[
\ord_i(M,b)\od\#(\{c\in [n]: M_{ic}\leq M_{ib}\}).
\]}
\State{{\bf Step 2:} Define $\hat{t}_{\max}(a)\od (1/n)\cdot\max_{i\in [m]} \ord_i(M, a)$.}
\State{{\bf Step 3:} Define an increasing filtration of simplicial complexes $\hat{\mathcal{K}}_a$ by 
\[
\hat{\mathcal{K}}_a\left(t\right)\od\Delta\left(\left\{\sigma_b:\sigma_b = \{i\in [m]:\ord_i(M,b)\leq\ord_i(M,a)-n (\hat{t}_{\max}(a)-t)\}\right\}_{b\in [n]}\right),
\]
where $t\in\{0, 1/n, 2/n, ..., \hat{t}_{\max}(a)\}$.}
\end{algorithmic}
\end{algorithm}

The next algorithm, Algorithm \ref{alg:Lk}, is for computing $L_k(M)$. Note that, in the algorithm, PersistenceIntervals is a function with two inputs, a filtration of simplicial complexes and a positive integer that is set to limit the dimension of the computation of persistent homology to avoid possible intractable computational complexities. As the name suggests, the output of PersistenceIntervals is the persistence intervals of the first input in dimensions less than or equal to the second input.
\begin{algorithm}[H]
\caption{Computation of $L_k(M)$}\label{alg:Lk}
\begin{algorithmic}
\State{}
\State{{\bf INPUTS:}}
\State{(1) $M$: an $m\times n$ real data matrix.}
\State{(2) $d_{\up}$: a positive integer, used to limit the dimension of the computation of persistent homology; namely, the persistent homology is only computed for dimension $0, 1, ..., d_{\up}$ to make it computationally feasible.}
\State{}
\State{{\bf OUTPUT:}}
\State{$[L_k(M)\text{ for }k = 0, 1, ..., d_{\up}]$: an array of nonnegative real numbers.}
\State{}
\State{{\bf STEPS:}}
\State{{\bf Step 1:} For $a\in [n]$, compute
\begin{align*}
\mathcal{I}_a& \od \PersistenceIntervals\left(\hat{\mathcal{K}}_a, d_{\up}\right)\od \left\{\dgm\left(H_k\left(\hat{\mathcal{K}}_a\right)\right)\text{ for } k = 0,1, ..., d_{\up}\right\}\text{, and} \\
\mathcal{L}_{a, \max}& \od \left\{\max\left\{\beta-\alpha:(\alpha,\beta)\in \dgm\left(H_k\left(\hat{\mathcal{K}}_a\right)\right)\right\}\text{ for } k = 0,1, ..., d_{\up}\right\}\\
& \od \left\{l_{\max}\left(k,\hat{\mathcal{K}}_a\right)\text{ for } k = 0, ..., d_{\up}\right\}.
\end{align*}
where $\PersistenceIntervals(\hat{\mathcal{K}}_a, d_{\up})$ computes the persistence intervals of $\hat{\mathcal{K}}_a$ in dimensions up to $d_{\up}$.}
\State{{\bf Step 2:} For $k = 0, 1, ..., d_{\up}$, compute $L_k(M)$ by 
\[
L_k(M)\od\max\{l_{\max}(k,\hat{\mathcal{K}}_a):a\in [n]\}.
\]}
\end{algorithmic}
\end{algorithm}

\subsection{How to use the algorithms under different situations}\label{sec:use_L_k_stat_subsampling}
The worst case complexity of a standard algorithm for computing the persistent homology of a 1-dimensional filtration of simplicial complexes is cubical in the number of simplices (see, e.g. Section 5.3.1 in \cite{Otter_2017} and references therein). Since each $\hat{\mathcal{K}}_a$ starts from the empty simplicial complex and ends at the full simplex $\Delta^{m-1}$, we would need to go through all faces of $\Delta^{m-1}$. However, since we limit the computation only in dimension $0, 1, ..., d_{\up}$, where $d_{\up}\leq m-1$ is pre-set, we only need to consider the $(\min\{d_{\up}+1, m-1\})$-skeleton of $\Delta^{m-1}$. Therefore, for our algorithm, the number of faces in the 1-dimensional filtration is
\[
\sum_{k = 0}^{\min\{d_{\up}+2, m\}} {m\choose k}
\]
which is $O(m^{d_{\up}+2})$. Since there are $n$ such $a\in [n]$, the worst case complexity of computing $\{L_k(M)\}_{k = 0}^{d_\up}$ is $O(n\cdot m^{3(d_{\up}+2)}) = O(n\cdot m^{3d_{\up}+6})$, which is of degree $3d_{\up}+6$ in $m$ but only linear in $n$. 

Since the algorithm is linear in $n$, even in the case when $n$ is large, as long as $m$ is not too large, the algorithm is still tractable. Moreover, to use the full power of Theorem \ref{thm_asym_L_k}, we would want $n$ to be large. In the case when $n$ is large, we may {\it subsample the points} (i.e. the columns) to see how large the variance of $L_k(M_n)$ is; this is called {\it bootstrap} in statistics. Moreover, we can implement the subsampling for different numbers of columns and get the convergence trend. 

On the other hand, to infer the dimension $d\RegPair$, we will need at least $m\geq d\RegPair+1$. Thus, we want $m$ to be not too small. However, since the computational complexity of $L_k(M_n)$ goes up in high degree order with respect to $m$, we cannot have $m$ being too large. In the case when $m$ is too large, we can overcome the computational difficulty by {\it subsampling the functions} (i.e. the rows); namely, pick randomly $m_s$, say $m_s = 10$, functions, which correspond to their respective $m_s$ rows of $M_n$ and compute the $L_k$ of the submatrix thus formed; repeat this process many times and see how the result is distributed. 

We elaborate on these two methods (subsampling points or functions) in the following two subsections. We also implement the methods for estimating the embedding dimension in their appropriate situations, plot the results and give some principles for decision making (i.e. deciding, given the the plot and $k$, whether we accept $L_k\RegPair>0$ or not).

\subsubsection{Subsample points when $n$ is sufficiently large}
In the case when $n$ is sufficiently large, say $n\geq 300$, we are allowed to subsample, say $n_s$, points (i.e. columns of $M_n$) and obtain the {\it variance information}. Moreover, letting $n_s$ go up, we can obtain further how the trend of convergence goes, which, by Theorem \ref{thm_asym_L_k}, should converge to the true $L_k\RegPair$. The technique of subsampling is called {\it bootstrap} in statistics. 

Figure \ref{fig:boxplot_L_k_M_n} is the boxplots\footnote{The {\it boxplot} of a collection of real numbers is a box together with a upper whisker and a lower whisker attached to the top and bottom of the box and possibly some dots on top of the upper whisker or below the lower whisker. From the box part, one can read out the {\it first quartile} $Q1$ (25th percentile), {\it medium} $Q2$ (50th percentile) and {\it third quartile} $Q3$ (75th percentile) which are the bottom end, line in-between, and the top end of the box. The value $Q3-Q1$ is called the {\it interquartile range (IQR)}. The lower whisker and upper whisker, resp., label the values $Q1-1.5IQR$ and $Q3+1.5IQR$, resp. Values outside of the whiskers are regarded as outliers and labelled by dots.\label{footnote:boxplot}} of $L_k(M_{n_s})$ obtained by implementing this idea under different settings of $(d,m,n)$, where $d=d\RegPair$ is the dimension of $\RegPair$, $m$ is the number of functions and $n$ is the number of data points. Here, we choose $(m, n)$ to be $(10, 350)$ throughout, where $m$ is moderate for computation and $n$ is sufficiently large for subsampling. Subsampling is repeated $100$ times for each boxplot. To compare with the result of a {\it purely random} matrix, we also generate a $10\times 350$ matrix whose entries are iid from $\Unif(0,1)$ and compute its $L_k$'s. The details of how the boxplots are generated are in the caption of Figure \ref{fig:boxplot_L_k_M_n}.

\begin{figure}[H]
  \centering
  \includegraphics[width=0.9\linewidth]{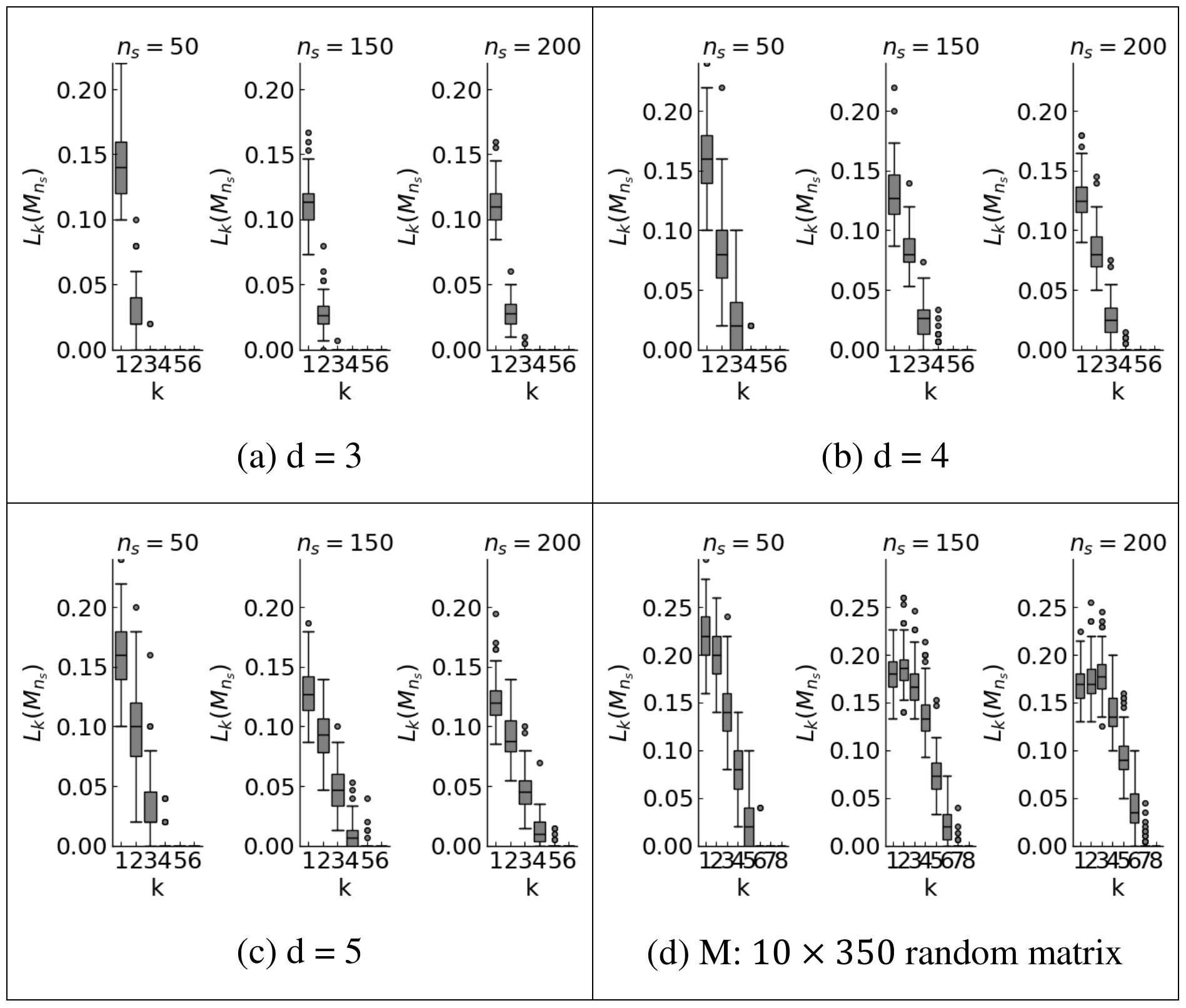}
  \caption{The four panels are boxplots of $L_k$ obtained from subsampling the points. Throughout the panels, $m = 10$ and $n = 350$, where $m$ is the number of functions and $n$ is the number of total sample poionts. The panels correspond to (a) $d=3$, (b) $d=4$, (c) $d=5$, where $d=d\RegPair$ is the dimension of $\RegPair$. The functions are chosen to be random quadratic functions defined on the unit $d$-ball in $\R^d$. Panel (d) is obtained by computing the $L_k$'s of an $m\times n$ matrix $M_n$ with entries i.i.d. from Unif(0,1), which is treated as a purely random matrix and whose main purpose is for comparison with other panels. Each figure in each panel is generated by subsampling $n_s = 50, 150, 200$ columns of $M_n$, repeated $100$ times. By the decision principle, every figure in panel (a) successfully infer their respective true dimensions; in panel (b), $n_s=50$ fails to infer the true dimension $4$ but only infers a lower bound $3$ while $n_s = 150, 200$ successfully infer the true dimension $4$; in panel (c), both $n_s = 50, 150$ fail to infer the true dimension $5$ but only infer a lower bound $4$ while $n_s = 200$ successfully infer the true dimension $5$. In panel (d), the figure has a quite different behavior.\protect\footnotemark}
  \label{fig:boxplot_L_k_M_n}
\end{figure}
\footnotetext{It can be proved that, if entries of $M_n\in\mathcal{M}_{m,n}^o$ are i.i.d. from $\Unif(0,1)$, then $L_k(M_n)$ behaves as positive for $k \leq m-2$ and $L_k(M_n)$ behaves as going to zero, for $k\geq m-1$; thus $L_k(M_n)$ relies on $m$ instead of an intrinsic $d$.}

Let us elaborate a little more on Figure \ref{fig:boxplot_L_k_M_n}. The {\bf decision principle} we propose to follow is that, 
\begin{quote}
{\it on each boxplot of $L_k(M_{n_s})$, if the first quartile\footnote{i.e. the 25th percentile} $Q1$ is not greater than $0$, reject $L_k\RegPair>0$; otherwise, accept $L_k\RegPair>0$}. 
\end{quote}

For panel (a) where $d = 3$, we can see that, as $n_s$ goes up, the variance of $L_k(M_{n_s})$ for each $k$ goes down. For $k = 2$, the first quartile of $L_k(M_{n_s})$ is greater than $0$ even for $n_s = 50$; for $k\geq 3$, $L_k(M_n)$ stays at $0$ with only some noise-like dots all the time. According to this principle, we can conclude $d\geq 3$ for this regular pair. In fact, as we know in advance, $d = 3$. 

For panel (b) where $d = 4$, the same shrinking variance behavior can be observed. Moreover, the principle concludes $d\geq 4$ after $n_s = 150$, where the first quartile starts to stay away from $0$. Similarly, for panel (c) where $d = 5$, in $n_s = 50$ and $n_s = 150$, our principle concludes $d\geq 4$ and in $n_s = 200$, it concludes $d\geq 5$. 

It is observed that, for higher $d$, we would need $n_s$ to be larger to make the best conclusion (i.e. inferring the true dimension). However, by making $n_s$ go up, the variance of $L_k(M_{n_s})$ goes down and we may also use this information. Therefore, for small sample case, one may count on this convergence behavior and develop other principles by quantifying the trend of convergence. For example, in panel (c) where $d = 5$, when $n_s$ goes up from $50$ to $150$, we observe that $L_4(M_{n_s})$ pokes out from noiselike outliers to a filled box. This trend suggests that we ``may accept" $L_4\RegPair>0$. We will leave it to the practitioners to decide their own principles on how to use the convergence trend information in their fields of interest. 

\subsubsection{Subsample functions when $m$ is large}
As we mentioned earlier, the worst case computational complexity of $L_k(M_n)$ goes up although polynomially but with degree $3 d_\up+6$ (high degree) in $m$, the number of rows of $M_n$. To overcome this difficulty, we propose to {\it subsample the rows} (i.e. the collection of functions) of $M_n$. Specifically, for a fixed number $m_s<m$, we randomly choose $m_s$ rows of $M_n$ and construct the $m_s\times n$ submatrix $M_{m_s\times n}$ accordingly, compute $L_k(M_{m_s\times n})$ and repeat the process as many times as assigned. Figure \ref{fig:subsample_function} is the boxplot of $L_k(M_{m_s\times n})$ with $m_s = 10$, repeated $N_{\rep} = 1000$ times, under different settings. Notice that throughout the plots, $m = 100$, $n = 150$, $m_s = 10$ and $N_{\rep} = 1000$. We still adopt the principle as last subsection that we only accept $L_k\RegPair>0$ when the first quartile Q1 is above $0$. Therefore, the concluding lower bounds for the plots are $2,3,4$ and $4$, resp., for panel (a), (b), (c) and (d) in Figure \ref{fig:subsample_function}.

\begin{figure}[H]
	\centering
  \includegraphics[width=0.9\linewidth]{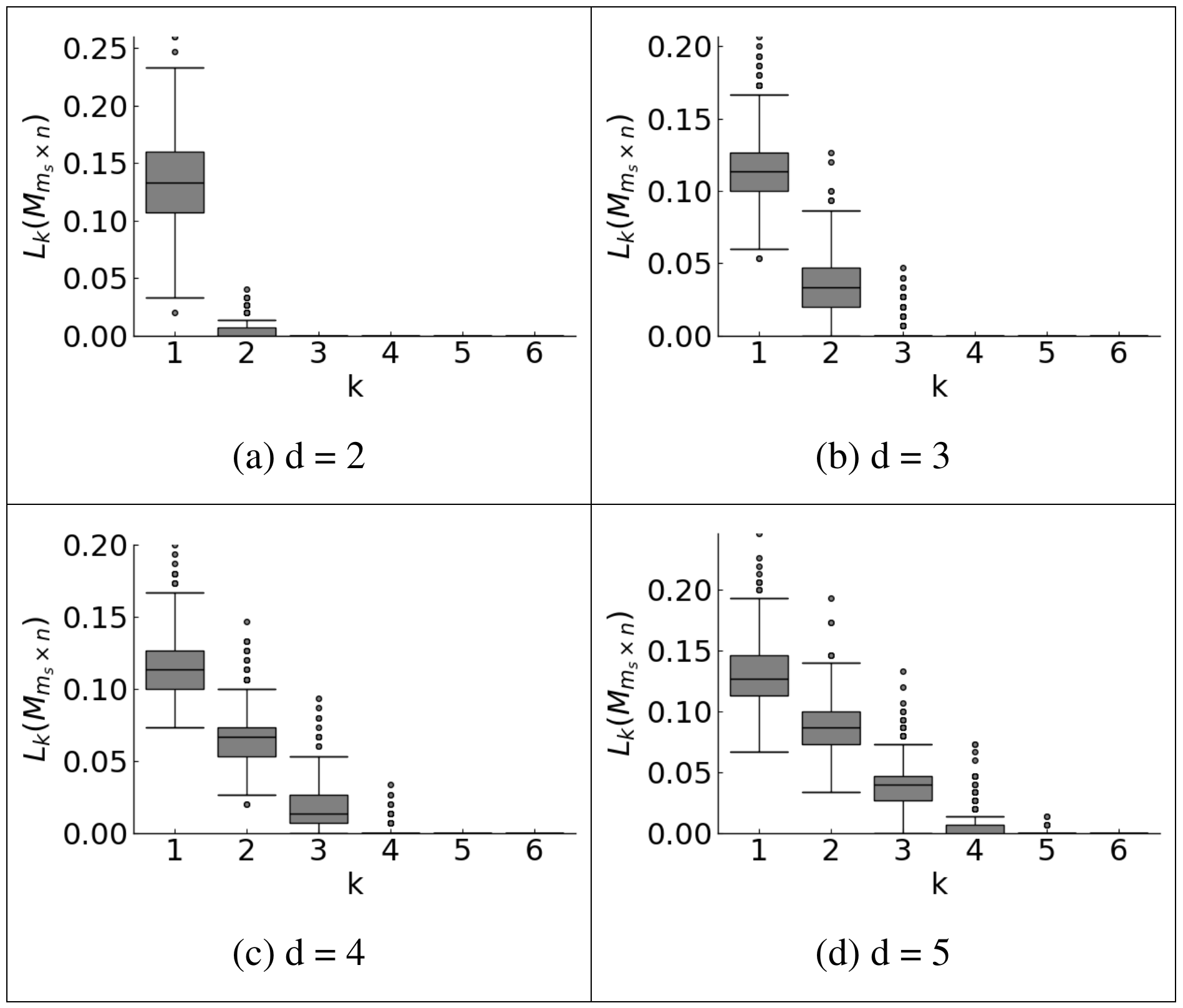}
  \caption{Boxplots of $L_k$ obtained from subsampling functions. Throughout the panels, $m = 60$, $n = 150$ and $m_s = 10$, where $m$ is the number of functions (rows of $M_n$), $n$ is the number of total sample points and $m_s$ is the number of functions used in each function subsampling. Each panel is generated under a fixed regular pair with (a) $d=2$, (b) $d=3$, (c) $d=4$, (d) $d=5$, where $d$ is the true dimension of the regular pair. Function subsampling is repeated $1000$ times for each panel.}\label{fig:subsample_function}
\end{figure}

A lower bound for $d\RegPair$ may not be very satisfactory. In Section 4 and 5, we develop some theory and methods to decide whether the lower bound obtained in this section is indeed the dimension $d\RegPair$.

%%%%%%%%%%%%%%%%%%%%%%%%%%%%%%%%%%%%%%%%%%%%%%%%%%%%%%
%%%%%%%%%%%%%%%   Section 4 
 
\section{$\hat{d}_\low(M,\varepsilon)$ as an asymptotically consistent dimension estimator in the class of complete regular pairs}
\label{sec:suff_cond}
We establish in Section 3
  that a lower bound $d_\low\RegPair$ of $d\RegPair$ is generally inferable from sampled  data.  Here  we provide a sufficient condition for  $d_\low\RegPair = d\RegPair$; this ensures that   the dimension $d\RegPair$ can be inferred with high probability. Recall that 
the   {\bf conic hull}  of a  set $S\subseteq\R^d$, denoted {\bf \boldmath$\cone(S)$}, is the set 
\begin{equation}\label{eq:cone} 
\cone(S)\od\left \{\sum_{i=1}^k c_iv_i \, \, \vert \,\,  c_1, ..., c_k\geq 0,\, \, v_1, ..., v_k\in S, \, \, k\in\N\right\}.
\end{equation}

\begin{definition}\label{def:Cent_1}
Let $\RegPair$ be a regular pair, where   $\mathcal{F} =  \{f_i\}_{i\in [m]} $ and each $f_i\colon K\to \mathbb R$ is differentiable. The set %$\Cent_1$ of $K$, defined by
\[
\Cent_1\od\left \{
x\in K\, \, \vert \,\,  
\cone
 \left( \{\nabla f_1(x), \dots , \nabla f_m(x)\}\right) = \R^d 
\right\},
\]
is called the    {\it type 1 central region} of $(\mathcal{F}, P_K)$. 
% A regular pair $\RegPair$ is said to be {\bf complete} if $ \Cent_1$ is not empty. 
\end{definition}  
%%%%%%%%%

\begin{definition}\label{defn:complete-regular-pair} A regular pair $\RegPair$ is said to be {\bf complete} if its $ \Cent_1$ is non-empty.  
\end{definition} 

It is perhaps intuitive  (see Figure \ref{fig:dim_infer_geo_idea} on page \pageref{fig:dim_infer_geo_idea}) that, for a \emph{sufficiently nice} complete regular pair, the lower bound in Lemma \ref{lem_lower_bound}   is indeed the dimension $d\RegPair$. More precisely, 

\begin{theorem}\label{thm_suff_cond_0}
Let $\RegPair$ be a regular pair, where   $\mathcal{F} =  \{f_i\}_{i\in [m]} $ and each $f_i\colon  K\to \mathbb R$ is differentiable. If $\RegPair$ is complete, then the lower bound in Lemma \ref{lem_lower_bound} is indeed the dimension of the regular pair, i.e. $d_\low \RegPair= d\RegPair$.
\end{theorem}
%%%%%%%%%%

\bigskip 
The proof is given in Section \ref{sec:proof: thm_suff_cond_0}. An immediate corollary of the above theorem and Corollary \ref{cor:d_low_estimator}, is the following 
\begin{theorem}\label{thm:d_low_estimate_true_d}
Let $\RegPair$ be a regular pair satisfying the conditions in Theorem \ref{thm_asym_L_k}, where   $\mathcal{F} =  \{f_i\}_{i\in [m]} $ and each $f_i\colon  K\to \mathbb R$ is differentiable. 
If $\RegPair$ is complete regular pair with dimension $d=d\RegPair$,  and matrices $M_n\in \mathcal{M}_{m,n}^o $ are  sampled from $\RegPair$,   then for every $\varepsilon\in\left  (0,L_{d-1}\RegPair\right)$, 
\begin{equation}
\lim_{n\to\infty} \Pr\left[\hat{d}_{\low}(M_n,\varepsilon) = d\RegPair\right] = 1.
\end{equation}
\end{theorem}
In other words, $\hat{d}(\varepsilon):\mathcal{M}_{m,n}^o\to\N$ defined by $\hat{d}(\varepsilon)(M_n)\od\hat{d}_{\low}(M_n,\varepsilon)$ is an asymptotically consistent estimator in the class of complete regular pairs.

\begin{proof}
By Theorem \ref{thm_suff_cond_0}, $d_\low\RegPair = d\RegPair$. Moreover, by Corollay \ref{cor:d_low_estimator},\[
\lim_{n\to\infty} \Pr\left[\hat{d}_{\low}(M_n,\varepsilon) = d_{\low}\RegPair\right] = 1.
\]
Thus, the result follows. 
\end{proof}

\subsection{Proof of Theorem \ref{thm_suff_cond_0}}
\label{sec:proof: thm_suff_cond_0}

Recall, the following notation from Section \ref{sec:app_to_dim_infer_prob}.  Let  $\RegPair = (\{f_i\}_{i\in [m]},P_K)$ be a regular pair;    for any  $t\in [0,1]$, we denote 
$K^{(i)}(t) \od f_i^{-1}(-\infty,\lambda_i(t))$, 
where $\lambda_i(t) $ is a monotone-increasing function that satisfies $P_K(f_i^{-1}(-\infty,\lambda_i(t))) = t$.
  For any  $x\in K$, we also denote  $$T_i(x) \od P_K(f_i^{-1}(-\infty,f_i(x))).$$

\noindent Theorem \ref{thm_suff_cond_0} follows from  the following key lemma. 
\begin{lemma}\label{lem_cent_0_local} Let $(\{f_i\}_{i=1}^m, P_K)$ be a complete regular pair. Suppose  $x_0\in \Cent_1$,  then there exists $\varepsilon>0$ such that 
\begin{equation} \label{eq:nervex0}
\nerve\left(\left\{K^{(i)}(T_i(x_0)-t)\right\}_{i\in [m]}\right)\sim S^{d-1},\ \forall \  t\in [0, \varepsilon).
\end{equation} 
\end{lemma}

\begin{proof}[Proof of Theorem \ref{thm_suff_cond_0}] By Lemma \ref{lem_lower_bound}, $d_{\operatorname{low}}\RegPair\leq d\RegPair $.
We therefore only need to prove $L_{d-1}(\mathcal{F},P_K)>0$. Let $x_0\in\Cent_1$. By Lemma \ref{lem_cent_0_local},  $l_{\max}(d-1,x_0)>0$ (see Definition \ref{def_L_k}). Therefore, $L_{d-1}(\mathcal{F},P_K)>0$, completing the proof.
\end{proof}

\begin{proof}[Proof of Lemma \ref{lem_cent_0_local}] For each $i\in [m]$,  we denote by 
$$\Omega_i =  K^{(i)}(T_i(x_0))=\left \{ x \in K \, \vert \, f(x)<f(x_0) \right\} $$ 
the appropriate open convex sublevel set of $f_i$. Note that $ K^{(i)}(T_i(x_0)-t)\subseteq \Omega_i$ for any 
$t \geq 0$, and the nerve in the  left-hand-side of  \eqref{eq:nervex0} is a subcomplex of the    $\nerve\left(\left\{\Omega_i\right\}_{i\in [m]}\right)$.

For  each non-empty $\sigma\in \nerve\left(\left\{\Omega_i\right\}_{i\in [m]}\right)$, the   subset $\bigcap_{i\in\sigma}\Omega_i $ is open and non-empty and hence has a  nonzero $P_K$ measure. Thus\footnote{See Lemma \ref{lem:Gap_Lemma4_5} in Section \ref{App_Sec:Gap_Lemma4_5}.} there exists $\varepsilon_\sigma>0$, such that  $ \bigcap_{i\in\sigma} K^{(i)}(T_i(x_0)-t) $ is non-empty for any $t\in [0, \varepsilon_\sigma)$. Choosing $\varepsilon$ to be the minimum of all such $\varepsilon_\sigma$ thus guarantees that 
$$\nerve\left(\left\{K^{(i)}(T_i(x_0)-t)\right\}_{i\in [m]}\right)= \nerve\left(\left\{\Omega_i\right\}_{i\in [m]}\right)  \ \forall \  t\in [0, \varepsilon). $$

 It thus suffices to prove \eqref{eq:nervex0} for $t=0$. 
Since each  $\Omega_i$  is open and convex,  by the  nerve lemma\footnote{See, e.g., Theorem 10.7 in \cite{Bjorner:1996:TM:233228.233246} or Corollary 4G.3 in \cite{MR1867354}.}  it is enough  to show  that $\bigcup_{i\in [m]} \Omega_i \sim S^{d-1}$. Moreover, since  $x_0$ lies on the boundary of each $\Omega_i $, the union $  \{ x_0\} \cup \bigcup_{i\in [m]} \Omega_i $ is star-shaped.  Therefore, it suffices to prove that there exists $\eta>0$ such that 
\begin{equation} \label{eq:inclusion} 
B_\eta(x_0)\setminus\{x_0\}\subseteq \bigcup_{i\in [m]} \Omega_i, 
\end{equation} 
where $B_\eta(x_0) = \{x\in\R^d: \lVert x-x_0\rVert <\eta\}$. 

Suppose no such $\eta>0$ exists,  then, for all $n\in\N$, there exists a unit vector  $v_n\in S^{d-1}=\{ x\in \mathbb R^d \, , \, \Vert{x} \Vert =1\}$ such that $x_n=x_0+\frac{1}{n}\cdot v_n\notin \bigcup_{i\in [m]} \Omega_i$. By compactness of $S^{d-1}$, there is an infinite    subsequence $\{ v_{n_j}\}$, that converges to a particular  $v_*\in S^{d-1}$.
Since all $f_i$  are differentiable, using Taylor's theorem, we obtain
\begin{equation}\label{sphere_persist_Taylor}
n_j \left( {f_i(x_n) -f_i(x_0)} \right) 
=\langle \nabla f_i(x_0), v_{n_j}\rangle+O(1/n_j ).
\end{equation}
Since $x_n \notin \Omega_i $ for all $i$, $f_i(x_n)\geq f_i(x_0)$. Taking $\liminf_{j\to\infty}$ on both sides of equation (\ref{sphere_persist_Taylor}), we conclude that 
$
\langle \nabla f_i(x_0), v_*\rangle\geq 0, \ \forall\ i\in [m]
$. 
Since $-v_*\in \cone(\{\nabla f_i(x_0)\}_{i\in [m]}) = \R^d$,  choosing appropriate nonnegative coefficients in \eqref{eq:cone} yields  $0\leq \langle -v_*, v_*\rangle$ and hence $v_* = 0$, a contradiction. Therefore  the inclusion \eqref{eq:inclusion} holds for some $\eta>0$. 
\end{proof}

%%%%%%%%%%%%%%%%%%%%%%%%%%%%%%%%%%%%%%%%%%%%%%%%%%%%%%
%%%%%%%%%%%%%%%   Section 5
\section{Testing the completeness of $\RegPair$ from sampled  data}

 Theorem \ref{thm_suff_cond_0} establishes  that completeness of  $\RegPair$ implies $d_\low\RegPair = d\RegPair$, and thus the data dimension $d\RegPair $ can be inferred from sampled data.   Unfortunately, completeness cannot be directly tested from sampled data,  since the gradient information is not directly accessible from discrete samples. Here we consider a different notion of {\it central region},  $\Cent_0\subseteq K$, which, under some generic assumtion, is indistinguishable from $\Cent_1$ in the probability measure $P_K$ (Lemma \ref{lem_cent_inclusion}). We also establish that  the probability measure   of  $\Cent_0$  can be approximated from sampled data (Theorem \ref{thm_suff_cond_test}). This enables one  to test completeness of a regular pair from sampled data. 

%\subsection{Statements of Main Theorems}
%We start with the definition of $\Cent_0$.
\begin{definition}\label{def:Cent_0}
Let $(\mathcal{F}, P_K) $ be a regular pair, the subset  
\[
\Cent_0\od\left \{x\in K\, \vert\,  \bigcap_{i\in [m]} f_i^{-1}(-\infty,f_i(x)) = \varnothing\right \},
\]
is called the  {\it type 0 central region} of $(\mathcal{F}, P_K)$.
\end{definition}

% We  make  the following  genericity  assumption that guarantees  that  the  symmetric difference of $\Cent_1$ and $\Cent_0$  has  measure zero. 

\begin{definition}
A set of vectors $V = \{v_1, ..., v_m\}\subseteq\R^d$ is said to be {\bf in general direction} if, for every $\sigma\subseteq [m]$ with $|\sigma|\leq d$, the set of vectors $\{v_i\}_{i\in\sigma}$ is linearly independent. 
A collection of differentiable functions   $\mathcal{F} = \{f_i:K\to\R\}_{i\in [m]}$  is said to be {\bf in general position} if for (Lebesgue) almost every $x$ in $K$, the vectors  $\{\nabla f_i(x)\}_{i\in [m]}$ are in general direction.
\end{definition}

\begin{lemma}\label{lem_cent_inclusion}
Let $(\mathcal{F}, P_K)$ be a regular pair, where each function in $ \mathcal{F}$ is differentiable. Assume that $\mathcal{F}$ is in general position, then 
\begin{equation} \label{eq-lem-5.3} 
P_K(\Cent_1\setminus\Cent_0) =P_K(\Cent_0\setminus\Cent_1) = 0. \end{equation} 
\end{lemma}

%%%%%%%%%%%%%%%%%%%%%%% The following texts are hidden
\iffalse
Combining with Theorem \ref{thm_suff_cond_0}, we have an immediate corollary.
\begin{corollary}\label{cor_suff_cond_1}
Let $\RegPair = (\{f_i\}_{i\in [m]},P_K)$ be a regular pair, where each $f_i:K\subseteq\R^d\to\R$ is differentiable. Assume that $\{x\in K:\nabla f_i(x) = 0\text{ for some }i\}$ has (Lebesgue) measure zero and that $\mathcal{F}$ is in general position. If $P_K(\Cent_0)>0$, then $d_\low\RegPair = d\RegPair$.
\end{corollary}
In words, Corollary \ref{cor_suff_cond_1} states that  $P_K(\Cent_0)>0$ is a sufficient condition that guarantees $d_\low = d$. 
\fi
%%%%%%%%%%%%%%%%%%%%%%% The above texts are hidden
The proof is given in Section \ref{subsection5.1}. 
It can be shown that $\Cent_1$ of a regular pair is   an open set (see Lemma \ref{lem:cent_0_open}  in the Appendix).  Thus  completeness of a regular pair $\RegPair$   is equivalent to $P_K(\Cent_1)>0$. Lemma \ref{lem_cent_inclusion} ensures that  completeness of  a regular pair   in general position is equivalent to  $P_K(\Cent_0)>0$. In order to test whether $P_K(\Cent_0)>0$, one can use the following natural discretization. 

\begin{definition}\label{def_est_cent_1}
For a matrix $M\in\mathcal{M}_{m,n}^o$, the set  
\[
\widehat{\Cent}_0(M)
\od\left \{
a\in [n]:\bigcap_{i\in [m]}\left \{b\in [n]: M_{ib} <M_{i a}\right \}  = \varnothing
\right\}
\]
is called the {\it discretized central region}.
\end{definition}

If  a  matrix $M\in\mathcal{M}_{m,n}^o$ is sampled from a regular pair, then for each $a\in [n]$, the set $ \{b\in [n]: M_{ib} <M_{i a} \} $ is a  discretization of $f_i^{-1}(-\infty,f_i(x_a))$, and   $\widehat{\Cent}_0(M)$ can be thought of as an approximation of $\Cent_0$. The following theorem confirms this intuition.
 
\begin{theorem}\label{thm_suff_cond_test}
Let $M_n\in\mathcal{M}_{m,n}^o$  be sampled from a regular pair, then $\frac1n{\#(\widehat{\Cent}_0(M_n))} $ converges to $P_K(\Cent_0)$ in probability:    
\[ \text{ for all } \epsilon>0,\quad 
\lim_{n\to\infty} 
\Pr\left[\left|\frac1{n}{\#(\widehat{\Cent}_0(M_n))}-P_K(\Cent_0)\right|>\epsilon\right]= 0.
\].
\end{theorem}
The proof  involves technicalities used in proving the Interleaving Convergence Theorem (Theorem \ref{thm_inter_con}) and is given in Section \ref{sec:append_proof_suff_con} in the Appendix.
Theorem \ref{thm_suff_cond_test} establishes that  $\frac1{n}{\#(\widehat{\Cent}_0(M_n))}$ serves as an  approximation of $P_K(\Cent_0)$,  and thus enables one to  to test whether $P_K(\Cent_0)>0$. Thus, by Lemma \ref{lem_cent_inclusion},  this provides a way to test the completeness of  the underlying regular pair $\RegPair$.

\subsection{Proof of Lemma \ref{lem_cent_inclusion}}
\label{subsection5.1}
\noindent First we prove the first part of Lemma \ref{lem_cent_inclusion}.

\begin{lemma} \label{lem56} 
Let  $(\mathcal{F},P_K) = (\{f_i\}_{i\in [m]},P_K)$ be a regular pair, where each function in $ \mathcal{F}$ is differentiable. Assume that $\mathcal{F}$ is in general position, then  $P_K(\Cent_1\setminus \Cent_0)=0$.
\end{lemma} 
\begin{proof} Let   $K^\prime=\{ x\in K\,  \vert \, \exists\  i, \,\nabla f_i(x)= 0 \} $ 
denote the union of critical points of functions in $\mathcal{F}$. Since $\mathcal{F}$ is in general position, $K'$ has Lebesgue measure zero. Assume $x_0\in\Cent_1\setminus K^\prime$,  and thus 
\begin{equation} \label{eq:C1}
\forall u\in \mathbb R^d  \quad \exists\ c_1,\dots,c_m \geq 0 \text{ such that } 
u=\sum_{i=1}^m c_i \nabla f_i(x_0).
\end{equation}  
It can be easily shown, see e.g. Theorem 3.2.3 in \cite{GeneralConvexityBook}, that if  $f$ is differentiable and quasi-convex on an open convex $K$ with $\nabla f(x_0)\neq 0$, then $f(x)<f(x_0)$ implies $\left \langle \nabla f(x_0),x-x_0\right \rangle<0$.
Thus
\[
\bigcap_{i\in [m]} f_i^{-1}(-\infty,f_i(x_0))\subseteq \bigcap_{i\in [m]} \left \{x \,\vert \,   \left \langle   \nabla f_i(x_0),x-x_0 \right \rangle <0\right \} =\varnothing,
\]
where the last equality follows from  \eqref{eq:C1}, as one can chose $u=x-x_0$. This implies $x_0\in\Cent_0$. Therefore, $\Cent_1\setminus K'\subseteq\Cent_0$ and $P_K(\Cent_1\setminus\Cent_0)\leq P_K(K') = 0$. 
\end{proof}

\medskip 
 To prove the second half of Lemma  \ref{lem_cent_inclusion}, we first recall 
 that a convex cone $\mathcal{C}  \subseteq \mathbb R^d$ is called {\bf flat} if there exists $w\neq 0$ such that both $w\in \mathcal{C}$ and $-w\in \mathcal{C} $. Otherwise, it is called {\bf salient}. If a convex cone $\mathcal{C}$ is closed and salient, then there exists\footnote{Salient cones are also called  {\it pointed cones}. It is well-known (see, e.g.,  Section 2.6.1 in \cite{Boyd04convexoptimization}) that if $\mathcal{C}\subset\R^d$ is a closed salient cone, then its dual cone $\mathcal{C}^*\od\{w\in\R^d:\langle w,u\rangle\geq 0,\ \forall\ u\in \mathcal{C}\}$ has nonempty interior. Consider $-\mathcal{C}^*=\{w\in\R^d:\langle w,u\rangle\leq 0,\ \forall\ u\in \mathcal{C}\}=\{w\in\R^d:\langle w,u\rangle< 0,\ \forall\ u\in \mathcal{C}\}\cup\{w\in\R^d:\langle w,u\rangle= 0,\ \forall\ u\in \mathcal{C}\}$. If $\mathcal{C}$ is closed and salient, then $-\mathcal{C}^*$ has nonempty interior. Note that, if $d>0$, then $\{w\in\R^d:\langle w,u\rangle= 0,\ \forall\ u\in \mathcal C\}$ has measure $0$ and hence $\{w\in\R^d:\langle w,u\rangle< 0,\ \forall\ u\in \mathcal{C}\}$ is nonempty and any vector in it satisfies the wanted property.}
  $w\in  \mathbb R^d$ such that $\langle u, w\rangle<0$,  for all non-zero $ u\in \mathcal{C}$. 

\begin{lemma} \label{lem_basic_containment}
Let $(\mathcal{F},P_K) = (\{f_i\}_{i\in [m]},P_K)$ be a regular pair, where each $f_i$ is differentiable, then 
\begin{equation}  \label{eq:c0:minus:c1} 
\Cent_0\setminus\Cent_1 \subseteq \left \{x\in K:\cone\left(\left\{\nabla f_i(x)\right\}_{i\in [m]}\right)\text{ is flat but not $\R^d$}\right\}.
\end{equation} 
\end{lemma} 
\begin{proof} Let $x_0\in\Cent_0\setminus\Cent_1$. Denote 
 $$\mathcal{C}_0 \od \cone\left (\left \{\nabla f_i(x_0)\right \}_{i\in [m]}\right).$$ 
 Since $x_0\notin\Cent_1$, $\mathcal{C}_0  \neq\R^d$. Thus, it suffices to prove
that the  cone $\mathcal{C}_0$  is flat.  Suppose that the cone $\mathcal{C}_0 $ is not flat, then there exists $w$ such that $\langle u, w\rangle<0$,  for all non-zero $ u\in \mathcal{C}_0$. In particular, $\langle \nabla f_i(x_0), w\rangle <0,\ \forall\ i\in [m]$. 
Let us show  that  $\forall\ i\in [m]$, there exists $\alpha_i>0$ such that $x_0+\alpha_i w\in f_i^{-1}(-\infty,f_i(x_0))$. Suppose not, then  there exists  $i\in [m]$, such that   $f_i(x_0+\alpha w)\geq f(x_0),\ \forall\ \alpha>0$, and we have
\[
\langle \nabla f_i(x_0), w\rangle = \liminf_{\alpha\to 0^+} \frac{f(x_0+\alpha w)-f(x_0)}{\alpha}\geq 0,
\]
which is a contradiction. Thus, such positive  $\alpha_i$'s exist, and   we obtain that 
$$x_0+ \left( \min_{i\in [m]} \alpha_i \right)  w \in \bigcap_{i\in [m]} f_i^{-1}(-\infty,f_i(x_0)).$$ 
This contradicts the assumption that  $x_0\in\Cent_0$. Therefore the cone  $\mathcal{C}_0$ is flat. 
\end{proof}
It can be shown   that the inclusion in \eqref{eq:c0:minus:c1} is in fact an equality. However, since we do not need the equality here, it was  left out  the proof. 
To finish the proof of Lemma \ref{lem_cent_inclusion},  we use the following 
\begin{lemma}\label{General-Direction-Lemma}
Let $V = \{v_1, ..., v_m\}\subset \R^d$ be a set of vectors in general direction, then  $\cone(V)=\R^d$ or $\cone(V)$ is salient. 
\end{lemma}
 
To prove  Lemma \ref{General-Direction-Lemma}, we  use the following lemma. 
\begin{lemma}[see e.g. Theorem 2.5 in \cite{Regis2016}]\label{thm_flat_cone_equivalence}
Let   $V = \{v_1, ..., v_m\}$ be a set of non-zero vectors in $\R^d$,    then the following two statements are equivalent:
\begin{itemize}
\item[(i)] $\cone(V)=\sp(V)$;
\item[(ii)] For each $i\in [m]$, $-v_i\in\cone(V\setminus\{v_i\})$. 
\end{itemize}
\end{lemma}

\begin{proof}[Proof of Lemma \ref{General-Direction-Lemma}]
For $m\leq d$, the vectors  $\{v_1, ..., v_m\}$ are  linearly independent. Suppose there exists   $w\in R^d $,  such that $w, -w\in \cone(V)$. Thus there exist $a_i, b_i\geq 0$ with 
$ w= \sum_{i=1}^m a_iv_i=-\sum_{i=1}^m b_iv_i$. 
Since the vectors $\{v_1, ..., v_m\}$ are linearly independent,  $a_i+b_i=0$ for all $i\in [m]$, and thus $w=0$.  Therefore $\cone(V)$ is salient. 

For $m>d$, we  prove by induction on  the size of $V$. Suppose the result holds for any set of $m\geq d$ vectors in general direction. Let $V = \{v_1, ..., v_{m+1}\}$ be a set of $m+1$ vectors in general direction. Since any $d$ vectors in $V$ is a  basis in $\R^d$, $\sp(V) = \R^d$. Suppose the result is false for $V$; equivalently, $\cone(V)\neq \R^d$ and $\cone(V)$ is flat.  By Lemma \ref{thm_flat_cone_equivalence}, there exists $j\in [m+1]$ such that
\begin{equation}\label{not-in-cone}
-v_j\notin\cone(V\setminus\{v_j\}).
\end{equation}
Since $\cone(V)$ is flat, there exists a nonzero $w\in \R^d$ such that  $w, -w\in \cone(V)$, and thus     $w= \sum_{i=1}^{m+1}  a_iv_i=-\sum_{i=1}^{m+1}  b_iv_i $,     with $a_i, b_i\geq 0$ for all $i$. Let us prove that $a_j +b_j>0$. If $a_j+b_j = 0$, then $a_j = b_j = 0$. Thus, $w, -w\in \cone(V\setminus\{v_j\})$ and $\cone(V\setminus\{v_j\})$ is not salient. Since $|V\setminus\{v_j\}|=m$, by the induction hypothesis, we must have $\cone(V\setminus\{v_j\})=\R^d$. However, $\cone(V)\neq\R^d$ and hence $\cone(V\setminus\{v_j\})\neq\R^d$, a contradiction. Therefore  $a_j+b_j>0$, and we can conclude  that 
\[
-v_j = \sum_{i\in [m+1]\setminus\{j\}} \left(\frac{a_i+b_i}{a_j+b_j}\right) v_i\in\cone(V\setminus\{v_j\}),
\]
contradicting to \eqref{not-in-cone}. Therefore,   the result  holds for any $V$ in general direction  of size  $|V| = m+1$. This completes the proof by induction.
\end{proof}

\bigskip 
\noindent We now finish the proof of Lemma \ref{lem_cent_inclusion}. 
\begin{proof}[Proof of Lemma \ref{lem_cent_inclusion}]
The first half of the proof of Lemma \ref{lem_cent_inclusion} is done in Lemma \ref{lem56}. To prove the second half, we combine Lemma \ref{lem_basic_containment} and Lemma \ref{General-Direction-Lemma} to
obtain  
\begin{equation}\label{eq:general_direction_inclusion}
\Cent_0\setminus\Cent_1\subseteq \left\{x\in K:\left \{\nabla f_i\left(x\right)\right\}_{i\in [m]}\text{ is not in general direction}\right\}. 
\end{equation}
Since $\{f_i\}_{i\in [m]}$ is in general position, the right hand side of (\ref{eq:general_direction_inclusion}) has measure zero, completing the proof.
\end{proof}

 %%%%%%%%%%%%%%%%%%%%%%%%%%%%%%%%%%%%%%
%%%%%%%%%%%%%   Section 6 
\section{Appendix: proofs of the main theorems and supporting lemmas}
\label{section:Appendix}
\subsection{Proof of the dimension bound in Example \ref{ex_dim_obst}}
\label{subsection:dim:obstruction} 
\begin{proof}
For any $a\leq n-1$, the point  $x_a$  is ordered the last in the sequence $s_a=(\cdots, n,a) $;   thus,  by Lemma \ref{lem:convex_hull_noncontainment} each such point  $x_a$ cannot be in the interior of the convex hull of the other points, therefore  $x_n\in \conv(x_1, \dots , x_{n-1})$.   Assume  that the embedding dimension is $d\leq n-3$,  then  by the Caratheodory’s  theorem we conclude that there exists $b\in [n-1]$, such that 
\begin{equation} 
\label{eq:x_n_in_convhull} 
x_n\in \conv(x_1, ..., \hat{x}_b, ..., x_{n-1}).
\end{equation} 
However,  by assumptions \eqref{obstruction:example} there exists a continuous quasi-convex function  $f_b$ such that  $f_b(x_a)< f_b(x_{n})$ for all $a\in [n-1]\setminus\{ b\}$, thus Lemma \ref{lem:convex_hull_noncontainment}  yields a contradiction with \eqref{eq:x_n_in_convhull}.  Therefore, the matrix is not embeddable in dimension $d\leq n-3$.

To prove that these sequences are embeddable in dimension $d=n-2$, one can place points   $x_1, ..., x_{n-1}$ to the vertices of an $(n-2)$-simplex in $\R^{n-2}$, and  place $x_{n}$ to the barycenter of that simplex. 
By construction,    $\{x_1, ..., x_{n-1}\}$ are convexly independent and we have for following convex hull relations for every $i<n$: $x_n \notin\conv \left(\left\{x_1, ..., x_{n-1}\right \}\setminus\{x_i\}\right )$ , and 
$x_i \notin\conv(\{x_1, ..., x_{n}\}\setminus\{x_i\})$.  Therefore, by Lemma \ref{lem:convex_hull_noncontainment} there exist quasi-convex continuous functions that realize the sequences in \eqref{obstruction:example}. 
\end{proof}

%%%%%%%%%%%%%%%%%%

\subsection{Existence and Continuity of $\lambda_i(t)$ for $t\in (0,1)$}\label{AppSec:ell_i(t)}

\begin{lemma*}[Lemma \ref{lem:ell_t}]
Let $f:K\to\R$ be a continuous   function with $P_K(f^{-1}(\ell)) = 0$ for all $\ell\in\R$, where $P_K$ is a probability measure on a convex open set $K$ and $P_K$ is equivalent to the Lebesgue measure on $K$. Then there exists a unique strictly increasing continuous function $\lambda:(0,1)\to\R$ such that, for all $t\in (0,1)$, 
\begin{equation}
P_K\left(f^{-1}\left(-\infty,\lambda\left(t\right)\right)\right) = t.
\end{equation}
\end{lemma*}

\begin{proof}
Since $K$ is path-connected, by intermediate value theorem, $f(K)$ is an interval in $\R$. Define a function $p_f:f(K)\to [0,1]$ by $p_f(\ell) \od P_K(f^{-1}(-\infty,\ell))$. Rewriting Equation (\ref{eq:ell_t}) as $p_f(\lambda(t)) = t$, we note that $\lambda(t)$ (if exists) is the inverse of $p_f$, proving uniqueness of $\lambda(t)$. For the existence and continuity of $\lambda(t)$, it suffices to prove $p_f$ is continuous and strictly increasing. 

To prove $p_f$ is continuous, we prove $p_f$ is continuous from the right and from the left. Let $\ell\in f(K)$. For $\ell_n\nearrow\ell$ in $f(K)$,\footnote{Recall that a sequence $(\ell_n)_n\subset\R$ goes up to $\ell\in\R$, denoted $\ell_n\nearrow\ell$, if $\ell_n\leq\ell_{n+1}$, for all $n$, and $\lim_{n\to\infty} \ell_n = \ell$. $\ell_n\searrow\ell$ is similarly defined.} from definition, $f^{-1}(-\infty,\ell_n)\nearrow f^{-1}(-\infty,\ell)$.\footnote{Recall that a sequence of sets $(A_n)_n$ goes up to a set $A$, denoted $A_n\nearrow A$, if $A_n\subseteq A_{n+1}$, for all $n$, and $\bigcup_{n} A_n = A$. $A_n\searrow A$ is similarly defined.} Since $P_K$ is a finite measure, taking $P_K$ on both sides, we obtain $p_f(\ell_n)\nearrow p_f(\ell)$. 
Thus $p_f$ is continuous from the left. On the other hand, for $\ell_n\searrow\ell$ in $f(K)$, from definition, 
\[
\left(\bigcap_{n = 1}^\infty f^{-1}(-\infty,\ell_n)\right)\setminus f^{-1}(-\infty,\ell) = f^{-1}(\ell). 
\]
Thus $p_f(\ell_n)\searrow P_K(\bigcap_{n = 1}^\infty f^{-1}(-\infty,\ell_n))=P_K(f^{-1}(-\infty,\ell))+P_K(f^{-1}(\ell)) = p_f(\ell)+0 = p_f(\ell)$ and $p_f$ is continuous from the right. Therefore, $p_f$ is a continuous function. 

Now we turn to prove $p_f$ is strictly increasing. For $\ell_1<\ell_2$ in $f(K)$, we need to prove $p_f(\ell_1)<p_f(\ell_2)$. Let $U_1 = f^{-1}(-\infty,\ell_1)$ and $U_2 = f^{-1}(-\infty,\ell_2)$, which are open convex sets with $U_1\subseteq U_2$. Since $f(K)$ is an interval, for any $\ell\in (\ell_1,\ell_2)$, there exists $x\in K$ with $f(x) = \ell$. Thus $U_1\neq U_2$. Note that $U_2\nsubseteq cl(U_1)$; otherwise, $U_1\subseteq U_2 = int(U_2)\subseteq int(cl(U_1))=U_1$ will imply $U_1 = U_2$, where the last equality follows from openness and convexity of $U_1$. Thus, there exists $x_0\in U_2\setminus cl(U_1)$. Choose $\epsilon>0$ such that $B(x_0,\epsilon)\subseteq U_2$ but $B(x_0,\epsilon)\cap cl(U_1) = \varnothing$. Then $P_K(U_2)\geq P_K(U_1)+P_K(B(x_0,\epsilon))>P_K(U_1)$; equivalently, $p_f(\ell_2)>p_f(\ell_1)$. Hence, $p_f$ is strictly increasing. 
\end{proof}

\subsection{Proof of Interleaving Convergence Theorem\label{sec:pf_inter_conv_thm}}
The goal of this subsection is to prove Theorem \ref{thm_inter_con}, the Interleaving Convergence Theorem.
%%The flow chart of the dependencies of the Lemmas, Lemmas, and Theorems is in Figure \ref{fig:Thm_Flow_Chart}.%% 
The asymptotic behavior of $\Dow(S(M))$ actually follows from the asymptotic behaviors of several building blocks of $\Dow(S(M))$. We will first define these building blocks and prove their own asymptotic theorems and then put these asymptotic theorems together to prove the Interleaving Convergence Theorem. 

\iffalse %% THIS IS THE BEGINNING OF THE HIDDEN PARAGRAPH
\begin{figure}[h]
  \centering
  \includegraphics[width=1\linewidth]{Thm_Flow_Chart.pdf}
  \caption{Flow chart of lemmas, Lemmas, and theorems}	
  \label{fig:Thm_Flow_Chart}
\end{figure}

\fi %% THIS IS THE END OF THE HIDDEN PARAGRAPH

We start with an object that, as will be seen, can be used to express $\DDow(\mathcal{F},P_K)$. Recall that, for $t\in [0,1]$ and $i\in [m]$, 
\[
K^{(i)}(t)\od f_i^{-1}(-\infty,\lambda_i(t))
\]
where $\lambda_i(t)\in\R$ satisfies $P_K(f_i^{-1}(-\infty,\lambda_i(t))) = t$.
\begin{definition}\label{def_R_infty}
For a regular pair $(\mathcal{F}, P_K)$, define a function $R_\infty:[0,1]^m\to [0,1]$ by 
\[
R_\infty(t_1, ..., t_m)\od P_K\left(\bigcap_{i\in [m]} K^{(i)}(t_i)\right).
\]
\end{definition}

It is easy to see that $R_\infty$ is a {\it cumulative distribution function} ({\bf CDF}). We next introduce another CDF, denoted $R_n$, which will be used as an intermediate between $\DDow(\mathcal{F},P_K)$ and $\Dow(S(M))$. 

\begin{definition}\label{def_R_n}
For a point cloud $X_n\subset K$ of size $n$, sampled from a regular pair, we define a function $R_n:[0,1]^m\to [0,1]$ by 
\[
R_n(t_1, ..., t_m)\od \frac{1}{n}\cdot \#\left(X_n\cap \left(\bigcap_{i\in [m]} K^{(i)}(t_i)\right)\right).
\]
\end{definition}

For those familiar with nonparametric statistics, it is easy to see that $R_n$ is in fact the {\it empirical cumulative distribution function} ({\bf empirical CDF}) of $R_\infty$. However, $R_n$ is still not obtainable from the $m\times n$ data matrix $M_n = [f_i(x_a)]$ since $K^{(i)}(t_i)$ is not directly accessible from $M_n$. The next definition is introduced to solve this problem by considering a {\it step-function} approximation of $K^{(i)}(t_i)$.

\begin{definition}\label{def_K_n^(i)(t)}
Let $M_n\in\mathcal{M}_{m,n}^o$ be sampled from a regular pair $(\mathcal{F},P_K)$, where $\mathcal{F} = \{f_i\}_{i\in [m]}$. For $i\in [m]$ and $t\in [0,1]$, we define 
\[
K_n^{(i)}(t) = 
\begin{cases}
f_i^{-1}(-\infty, f_i(x_a)) &\text{ if $t\in\left[\frac{\ord_i(M_n,a)-1}{n}, \frac{\ord_i(M_n,a)}{n}\right)$, $a=1, ..., n$}\\
K &\text{ if $t=1$}
\end{cases}.
\]
\end{definition}

For a pictorial illustration of $K_n^{(i)}(t)$, please refer to Figure \ref{fig:K_n^(i)(t)}. Notice that there is a subscript $n$ in $K_n^{(i)}(t)$, indicating its dependence on the sampled matrix $M_n$.

\begin{figure}
  \centering
  \includegraphics[width=0.5\linewidth]{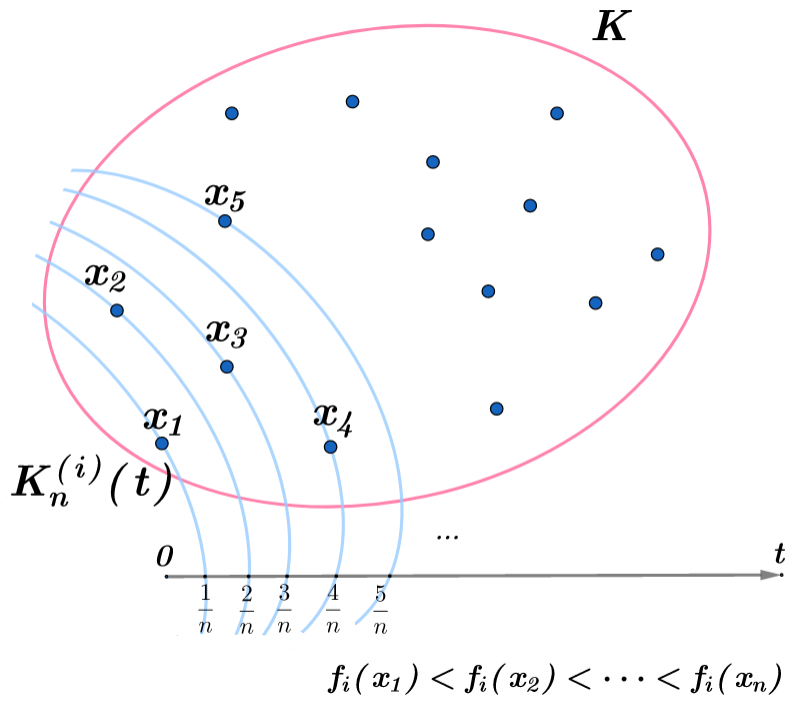}
  \caption{An illustration of $K_n^{(i)}(t)$ defined in Definition \ref{def_K_n^(i)(t)}.}	
  \label{fig:K_n^(i)(t)}
\end{figure}

The object in the next definition is obtainable solely from the data matrix $M_n$, sampled from a regular pair. 

\begin{definition}\label{def_hat_R_n}
Let $M_n=[M_{ia}]$ be an $m\times n$ data matrix, sampled from a regular pair. Define a function $\hat{R}_n:[0,1]^m\to [0,1]$ by 
\begin{align*}
\hat{R}_n(t_1, ..., t_m) &\od \frac{1}{n}\cdot \#\left(\left\{a\in [n]:\frac{\ord_i(M_n,a)}{n} \leq t_i,\ \forall\ i\in [m]\right\}\right)\\
 & = \frac{1}{n}\cdot \#\left(\bigcap_{i\in [m]}\left\{a\in [n]:\frac{\ord_i(M_n,a)}{n}\leq t_i \right\}\right).
\end{align*}
\end{definition}

In the following lemma, we rewrite $\hat{R}_n$ in a form that is similar to the definition of $R_n$, which helps build a connection between them. 
\begin{lemma}\label{lem_hat_R_n_expression}
Let $X_n = \{x_1,\cdots,x_n\}\subset K$ be a point cloud, sampled from a regular pair, and $M_n$ be the corresponding $m\times n$ data matrix. Then, for all $(t_1, ..., t_m)\in [0,1]^m$, 
\[
\hat{R}_n(t_1, ..., t_m)=\frac{1}{n}\cdot \#\left(X_n\cap\left(\bigcap_{i\in [m]} K_n^{(i)}(t_i) \right)\right).
\]
\end{lemma}
\begin{proof}
For each $i\in [m]$ and $t_i\in [0,1)$, there exists $b\in [n]$ such that $\frac{\ord_i(M_n,b)-1}{n}\leq t_i<\frac{\ord_i(M_n,b)}{n}$. Then $K_n^{(i)}(t_i) = f_i^{-1}(-\infty,f_i(x_b))$ and 
\begin{align*}
X_n\cap K_n^{(i)}(t_i) & = \{x_a\in X_n:f_i(x_a)<f_i(x_b)\}=\{x_a\in X_n:\ord_i(M_n,a)<\ord_i(M_n,b)\} \\
 & = \{x_a\in X_n:\ord_i(M_n,a)\leq nt_i\} = \left\{x_a\in X_n:\frac{\ord_i(M_n,a)}{n}\leq t_i\right\}.
\end{align*}
For $t_i = 1$, $K_n^{(i)}(t_i)=K$ and the above equality still holds. Thus, for any $(t_1, ..., t_m)\in [0,1]^m$, 
\[
X_n\cap\left(\bigcap_{i\in [m]} K_n^{(i)}(t_i)\right) = \bigcap_{i\in [m]}\left\{x_a\in X_n:\frac{\ord_i(M_n,a)}{n}\leq t_i\right\}.
\]
By the definition of $\hat{R}_n$, the equality follows.
\end{proof}

Now the intuition behind the approximations is quite clear: since $K_n^{(i)}$ is an approximation of $K^{(i)}$, by Lemma \ref{lem_hat_R_n_expression}, $\hat{R}_n$ is an approximation of $R_n$. Therefore, $\hat{R}_n$ also approximates $R_\infty$.

Next we connect $R_\infty$ and $\hat{R}_n$ with our target objects $\Dow(S(M))$ and $\DDow(\mathcal{F}, P_K)$. For simplicity, we introduce the following convenient notations:
\begin{definition}
For $(t_1,..., t_m)\in [0,1]^m$ and $\sigma\subseteq [m]$, define
\[
t^\sigma_i\od
\begin{cases}
t_i, \text{ if $i\in\sigma$}\\
1, \text{ otherwise}.
\end{cases}
\]
\end{definition}

With these notations, we have
\begin{theorem}\label{thm_Dow_DDow_express_in_hat_R_n_R_infty}
Let $\RegPair$ be a regular pair and $M_n$ be an $m\times n$ data matrix, sampled from $\RegPair$. Then, for all $(t_1, ..., t_m)\in [0,1]^m$, we have
\begin{itemize}
\item[(i)] $\DDow(\mathcal{F}, P_K)(t_1, ..., t_m) = \{\sigma\subseteq [m]:R_\infty(t_1^\sigma, ..., t_m^\sigma)\neq 0\}$.
\item[(ii)] $\Dow(S(M_n))(t_1, ..., t_m) = \{\sigma\subseteq [m]:\hat{R}_n(t_1^\sigma, ..., t_m^\sigma)\neq 0\}$.
\end{itemize}
\end{theorem}

\begin{proof}
For the first equality, recall that $\DDow(\mathcal{F}, P_K)(t_1, ..., t_m) = \nerve(\{K^{(i)}(t_i)\}_{i\in [m]})$; namely, $\sigma\in \DDow(\mathcal{F}, P_K)(t_1, ..., t_m)$ if and only if $\bigcap_{i\in\sigma} K^{(i)}(t_i)\neq\varnothing$. Since each $K^{(i)}(t_i)$ is open, $\sigma\in\DDow(\mathcal{F}, P_K)(t_1, ..., t_m)$ is also equivalent to $P_K(\bigcap_{i\in\sigma} K^{(i)}(t_i))\neq 0$. Notice that, by definition of $R_\infty$ and $t_i^\sigma$, 
\begin{align*}
R_\infty(t_1^\sigma, ..., t_m^\sigma)& = P_K\left(\bigcap_{i\in [m]} K^{(i)}(t_i^\sigma)\right)= P_K\left(\left(\bigcap_{i\in\sigma} K^{(i)}(t_i)\right)\cap\left(\bigcap_{i\in [m]\setminus\sigma} K^{(i)}(1)\right)\right) \\
& = P_K\left(\bigcap_{i\in\sigma} K^{(i)}(t_i)\right).
\end{align*}
Therefore, the first equality follows. 

For the second equality, recall, from Lemma \ref{lem_Dow(S)_nerve}, that 
\[
\Dow(S(M_n))(t_1,\cdots,t_m) = \nerve(\{A^{(i)}(t_i)\}_{i\in [m]})
\]
where $A^{(i)}(t_i) = \{a\in [n]:\frac{1}{n}\cdot\#(\{b\in [n]:b\leq_i a\})\leq t_i\}$. Thus, $\sigma\in\Dow(S(M))(t_1,\cdots,t_m)$ is equivalent to $\bigcap_{i\in\sigma} A^{(i)}(t_i)\neq\varnothing$, which is also equivalent to $\#(\bigcap_{i\in\sigma} A^{(i)}(t_i))\neq 0$ since each $A^{(i)}(t_i)$ is a finite set. Note that $A^{(i)}(t_i)$ is exactly $X_n\cap K_n^{(i)}(t_i)$. Thus, by Lemma \ref{lem_hat_R_n_expression}, 
\begin{align*}
\hat{R}_n(t_1^\sigma,\cdots,t_m^\sigma) 
&= \frac{1}{n}\cdot \#\left(X_n\cap\left(\bigcap_{i\in [m]} K_n^{(i)}(t_i^\sigma)\right)\right)
= \frac{1}{n}\cdot \#\left(\left(\bigcap_{i\in\sigma} A^{(i)}(t_i)\right)\cap \left(\bigcap_{i\in [m]\setminus\sigma} A^{(i)}(1)\right)\right)\\
& = \frac{1}{n}\cdot \#\left(\bigcap_{i\in\sigma} A^{(i)}(t_i)\right). 
\end{align*}
Hence, $\#(\bigcap_{i\in\sigma} A^{(i)}(t_i))\neq 0$ is equivalent to $\hat{R}_n(t_1^\sigma,\cdots,t_m^\sigma)\neq 0$, and the second equality follows.
\end{proof}

Before diving into asymptotic results, let us look at one useful property of $R_\infty$.

\begin{lemma}\label{lem_R_infty_unif_cont}
The map $R_\infty$ in Definition \ref{def_R_infty} is uniformly continuous. 
\end{lemma}

\begin{proof}
Let $(t_1, ..., t_m), (t_1', ..., t_m')\in [0,1]^m$. Denoting the symmetric difference of any two sets $A$ and $B$ by $A\ominus B\od (A\setminus B)\cup (B\setminus A)$, then 
\begin{align*}
&\left|R_\infty(t_1, ..., t_m)-R_\infty(t_1', ..., t_m')\right|
=\left| P_K\left(\bigcap_{i\in [m]} K^{(i)}(t_i)\right)-P_K\left(\bigcap_{i\in [m]} K^{(i)}(t_i')\right)\right|\\
&\leq P_K\left(\left(\bigcap_{i\in [m]} K^{(i)}(t_i)\right)\ominus\left(\bigcap_{i\in [m]} K^{(i)}(t_i')\right)\right)
\leq P_K\left(\bigcup_{i\in [m]}\left(K^{(i)}(t_i)\ominus K^{(i)}(t_i')\right)\right)\\
&\leq \sum_{i\in [m]} P_K\left(K^{(i)}(t_i)\ominus K^{(i)}(t_i')\right) = \sum_{i\in [m]} |t_i-t_i'|. 
\end{align*}
Using this inequality, it is now easy to obtain that $R_\infty$ is uniformly continuous.  
\end{proof}

Now we arrive at a theorem that is key to the proof of Interleaving Convergence Theorem. In the rest of the discussion, we use {\bf w.h.p.} to refer to {\bf with high probability}. Namely, if we state, as $n\to\infty$, w.h.p., a sequence $(A_n)_{n=1}^\infty$ of events holds, then this means that, as $n\to\infty$, the probability $\Pr[A_n]$ approaches $1$.

\begin{theorem}{\bf (1st Asymptotic Theorem)}\label{thm_1st_asym}\\
The sup-norm $\lVert R_n-R_\infty\rVert_\infty$ converges to $0$ in probability. In other words, for any $\epsilon>0$,  
\[
\lim_{n\to\infty} \Pr\left[\lVert R_n-R_\infty\lVert_\infty\leq \epsilon\right]=1.
\]
\end{theorem}

For the proof, we recall an intuitive fact from probability theory:
\begin{quoting}
For a finite collection of $N$ events $A_{1,n}, \cdots, A_{N,n}$ that depends on $n = 1, 2, \cdots$, if, for each $i\in [N]$, $\lim_{n\to\infty} \Pr[A_{i,n}]=1$, then $\lim_{n\to\infty} \Pr[\bigcap_{i\in [N]} A_{i,n}]=1$.
\end{quoting}

\begin{proof}[Proof of Theorem \ref{thm_1st_asym}]
Let $(t_1, ..., t_m)\in [0,1]^m$. Let $I$ be the indicator function of $\bigcap_{i\in [m]} K^{(i)}(t_i)$. In other words, $I:K\to\{0,1\}$ is a function defined by 
\[
I(x) = \begin{cases}
1&\text{ if $x\in \bigcap_{i\in [m]} K^{(i)}(t_i)$}\\
0&\text{ otherwise}.
\end{cases}
\]
Notice that, since $(K,P_K)$ is a probability space (with Borel $\sigma$-algebra), $I$ is a random variable. Moreover, by Definition \ref{def_R_n}, if $I_1, ..., I_n$ are i.i.d copies of $I$, then 
\[
R_n(t_1, ..., t_m)=\frac{1}{n}\sum_{a=1}^n I_a.
\]
Let $p=P_K\left(\bigcap_{i\in [m]} K^{(i)}(t_i)\right)=R_\infty(t_1, ..., t_m)=\Pr[I = 1]$. By Chebyshev inequality, for any $\epsilon>0$, as $n\rightarrow\infty$, 
\[
\Pr\left[\left|\left(\frac{1}{n}\cdot\sum_{a=1}^n I_a\right)-p\right|>\epsilon\right]
\leq \frac{\Var\left(\frac{1}{n}\cdot \sum_{a=1}^n I_a\right)}{\epsilon^2}
=\frac{p(1-p)}{n\epsilon^2}\rightarrow 0.
\]
i.e. $\Pr\left[\left|R_n(t_1, ..., t_m)-R_\infty(t_1, ..., t_m)\right|>\epsilon\right]\rightarrow 0$. Thus we have obtained the pointwise convergence version of the result.

To prove uniform convergence, consider the following:
By Lemma \ref{lem_R_infty_unif_cont}, $R_\infty$ is uniformly continuous. Thus there exists $\delta>0$ such that, for all $(t_1, ..., t_m)$ and $(t_1', ..., t_m')$ with $\max_{i\in [m]} |t_i-t_i'|\leq\delta$, 
\[
\left|R_\infty(t_1, ..., t_m)-R_\infty(t_1', ..., t_m')\right|\leq\epsilon.
\]
Subdivide $[0,1]^m$ into finitely many ($m$-dimensional) rectangles of sides shorter than $\delta$. Let $V$ be the collection of all vertices of all rectangles in the subdivision. Since $V$ is a finite set, by the above pointwise result, as $n\to\infty$, w.h.p., 
\begin{equation}\label{eq:uniform_conv_on_rectangle_vertices}
\sup_{\vec{t}\in V} |R_n(\vec{t})-R_\infty(\vec{t})|\leq\epsilon.
\end{equation}
In other words, 
\[
\lim_{n\to\infty} \Pr\left[\sup_{\vec{t}\in V} |R_n(\vec{t})-R_\infty(\vec{t})|\leq\epsilon\right] = 1.
\]
We now claim that Equation (\ref{eq:uniform_conv_on_rectangle_vertices}) implies 
\begin{equation}\label{eq:uniform_conv}
\lVert R_n-R_\infty\rVert_\infty = \sup_{\vec{t}\in [0,1]^m} |R_n(\vec{t})-R_\infty(\vec{t})|\leq 2\epsilon.
\end{equation}
Let $\vec{t} = (t_1, ..., t_m)$ be an arbitrary element in $[0,1]^m$. Then $\vec{t}$ lies in some small rectangle of the subdivision. Let $\vec{t}_1$ and $\vec{t}_0$ be the unique maximum and minimum, respectively, in the rectangle. Then 
\begin{align*}
& R_\infty(\vec{t})-\epsilon\\
& \leq (R_\infty(\vec{t}_0)+\epsilon)-\epsilon = R_\infty(\vec{t}_0) & \text{(uniform continuity of $R_\infty$)}\\
& \leq R_n(\vec{t}_0)+\epsilon & \text{($\vec{t}_0$ is a vertex in the subdivision)}\\
& \leq R_n(\vec{t})+\epsilon & \text{($R_n$ is monotone and $\vec{t}_0\leq\vec{t}$)}\\
& \leq R_n(\vec{t}_1)+\epsilon & \text{($R_n$ is monotone and $\vec{t}\leq\vec{t}_1$)}\\
& \leq (R_\infty(\vec{t}_1)+\epsilon)+\epsilon & \text{($\vec{t}_1$ is a vertex in the subdivision)}\\
& \leq R_\infty(\vec{t})+\epsilon+\epsilon+\epsilon & \text{(uniform continuity of $R_\infty$)}.
\end{align*}
Thus, $R_\infty(\vec{t})-2 \epsilon\leq R_n(\vec{t})\leq R_\infty(\vec{t})+2\epsilon$. i.e. $|R_\infty(\vec{t})-R_n(\vec{t})|\leq 2\epsilon$. Since $\vec{t}\in [0,1]^m$ is arbitrary, Equation (\ref{eq:uniform_conv}) follows. In other words, 
\[
\lim_{n\to\infty} \Pr\left[\sup_{\vec{t}\in [0,1]^m} |R_n(\vec{t})-R_\infty(\vec{t})|\leq 2\epsilon\right] = 1.
\]
Rescaling $2\epsilon$ to $\epsilon$, the uniform result follows. 
\end{proof}

For people familiar with non-parametric statistics, it is immediate that Theorem \ref{thm_1st_asym} is a natural $m$-dimensional generalization of the standard {\it Glivenko-Cantelli theorem}\footnote{Glivenko-Cantelli theorem has been generalized in many aspects in different literatures and is closely related to the famous VC (Vapnik–Chervonenkis) theory in theoretical machine learning. See, for example, Chapter 12 of \cite{DGL_learning} for a detailed introduction that connects standard Glivenko-Cantelli theorem and the VC theory.} under specific conditions.

Recall that, for $x\in K$ and $i\in [m]$, $T_i(x)\od P_K\left(f_i^{-1}(-\infty,f_i(x))\right)$.

\begin{lemma}\label{lem_K^i_K^i_n_interleave}
For all $\epsilon>0$, as $n\rightarrow\infty$, w.h.p., 
\[
K^{(i)}(t-\epsilon)\subseteq K_n^{(i)}(t)\subseteq K^{(i)}(t+\epsilon),\ \forall\ i\in [m],\ t\in[0,1]
\]
where $K^{(i)}(t)\od\varnothing$ for $t<0$ and to be $K^{(i)}(t)\od K$ for $t>1$. 
\end{lemma}

\begin{proof}
For $t=1$, $K_n^{(i)}(t) = K = K^{(i)}(t+\epsilon)$ and the inclusions are clearly satisfied. Now let $t\in [0,1)$. Then $t\in [\frac{a}{n},\frac{a+1}{n})$ for some $a\in\{0, 1,\cdots,n-1\}$. W.L.O.G., assume $f_i(x_1)<\cdots<f_i(x_n)$. By definition of $K_n^{(i)}(t)$, $K_n^{(i)}(t) = K^{(i)}(T_i(x_{a+1}))$. By monotonicity of $K^{(i)}$, it suffices to prove that, w.h.p., $|T_i(x_{a+1})-t|<\epsilon$. Notice that $R_\infty(1,\cdots,1,T_i(x_{a+1}),1,\cdots,1) = T_i(x_{a+1})$ and $R_n(1,\cdots,1,T_i(x_{a+1}),1,\cdots,1) = \frac{1}{n}\cdot \#(K^{(i)}(T_i(x_{a+1}))\cap X_n)=\frac{a}{n}$. Choose $n$ large enough such that $\frac{1}{n}<\frac{\epsilon}{2}$. Then 
\begin{align*}
& |T_i(x_{a+1})-t|\leq |T_i(x_{a+1})-\frac{a}{n}|+|\frac{a}{n}-t|\\
& =|R_\infty(1,\cdots,1,T_i(x_{a+1}),1,\cdots,1)-R_n(1,\cdots,1,T_i(x_{a+1}),1,\cdots,1)+|\frac{a}{n}-t|.
\end{align*}
In the last expression, by Theorem \ref{thm_1st_asym}, w.h.p., the first term is less than $\frac{\epsilon}{2}$, not depending on $t$, and the second term is less than $\frac{\epsilon}{2}$ by our choice of sufficient large $n$. Thus, w.h.p., $|T_i(x_{a+1})-t|<\epsilon$, for all $t\in [0,1]$ and the result follows.
\end{proof}

\begin{corollary}\label{cor_R_n_hat_R_n_interleave}
For all $\epsilon>0$, as $n\rightarrow\infty$, w.h.p., 
\[
R_n(t_1-\epsilon,...,t_m-\epsilon)
\leq \hat{R}_n(t_1,..., t_m)
\leq R_n(t_1+\epsilon,...,t_m+\epsilon)\ ,\ \forall\ (t_1, ..., t_m)\in [0,1]^m
\]
where, for the variables of $R_n$, negative inputs are automatically replaced by $0$ and inputs greater than $1$ are automatically replaced by $1$.   
\end{corollary}

\begin{proof}
As $n\to\infty$, w.h.p., 
\begin{align*}
R_n(t_1-\epsilon,...,t_m-\epsilon)& =\frac{1}{n}\cdot \#\left(\left(\bigcap_{i\in [m]} K^{(i)}(t_i-\epsilon)\right)\cap X_n\right)& \text{(Definition \ref{def_R_n})}\\
& \leq\frac{1}{n}\cdot\#\left(\left(\bigcap_{i\in [m]} K_n^{(i)}(t_i)\right)\cap X_n\right)& \text{(Lemma \ref{lem_K^i_K^i_n_interleave})}\\
&(=\hat{R}_n(t_1,..., t_m))& \text{(Lemma \ref{lem_hat_R_n_expression})}\\
& \leq\frac{1}{n}\cdot\#\left(\left(\bigcap_{i\in [m]} K^{(i)}(t_i+\epsilon)\right)\cap X_n\right)& \text{(Lemma \ref{lem_K^i_K^i_n_interleave})}\\
& = R_n(t_1+\epsilon,...,t_m+\epsilon)& \text{(Definition \ref{def_R_n})}.
\end{align*}
Hence the result follows.
\end{proof}

Motivated from Theorem \ref{thm_Dow_DDow_express_in_hat_R_n_R_infty}, the key to prove the Interleaving Convergence Theorem is the zero sets of $R_\infty$, $R_n$, and $\hat{R}_n$, explicitly defined below.
\begin{definition}\label{def:Z_infty_Z_n}
Let $R_\infty$, $R_n$, and $\hat{R}_n$ be defined as in Definition \ref{def_R_infty}, \ref{def_R_n}, and \ref{def_hat_R_n}. Define the following subsets of $[0,1]^m$:
\begin{itemize}
\item[(i)] $Z_\infty\od R_\infty^{-1}(0)=\left\{(t_1, ..., t_m):\bigcap_{i\in [m]} K^{(i)}(t_i) = \varnothing\right\}$.
\item[(ii)] $Z_n\od R_n^{-1}(0)=\left\{(t_1, ..., t_m):X_n\cap\left(\bigcap_{i\in [m]} K^{(i)}(t_i)\right)=\varnothing\right\}$.
\item[(iii)]$\hat{Z}_n \od \hat{R}_n^{-1}(0) = \left\{(t_1, ..., t_m):X_n\cap\left(\bigcap_{i\in [m]} K_n^{(i)}(t_i)\right)=\varnothing\right\}$.\footnote{See Definition \ref{def_K_n^(i)(t)} for the defining formula of $K_n^{(i)}(t_i)$.}
%% $= \left\{(t_1, ..., t_m):\bigcap_{i\in [m]}\left\{a\in [n]:\frac{\ord_i(M_n,a)}{n}\leq t_i \right\}=\varnothing\right\}$.
\end{itemize}
For $Z\subseteq\R^m$ and $\epsilon>0$, define 
\begin{equation}\label{eq:Z_plus_epsilon}
Z+\epsilon\od\{(t_1+\epsilon_1, ...,t_m+\epsilon_m):(t_1,...,t_m)\in Z,\ 0\leq\epsilon_i\leq\epsilon\}.
\end{equation}
\end{definition}

Note that, since $R_\infty$, $R_n$, and $\hat{R}_n$ are all monotone, $Z_\infty$, $Z_n$, and $\hat{Z}_n$ are closed under lower parital order; namely, for $Z = Z_\infty, Z_n\text{ or }\hat{Z}_n$, if $\vec{t}_*\in Z$, then $\vec{t}\in Z$ for all $\vec{t}\leq\vec{t}_*$.

\begin{lemma}\label{lem_Z_n_hat_Z_n_Z_infty_interleave}
Let $Z_\infty$, $Z_n$, and $\hat{Z}_n$ be defined as in Definition \ref{def:Z_infty_Z_n}. Then, for all $\epsilon>0$, as $n\rightarrow\infty$, w.h.p., 
\begin{itemize}
\item[(i)] $Z_n\subseteq \hat{Z}_n+\epsilon$ and $\hat{Z}_n\subseteq Z_n+\epsilon$.
\item[(ii)] $Z_n\subseteq Z_\infty+\epsilon$. 
\end{itemize}
Moreover, with probability $1$, 
\begin{itemize}
\item[(iii)] $Z_\infty\subseteq Z_n$.
\end{itemize}
\end{lemma}

\begin{proof} For the first inclusion in (i), let $z\in Z_n$, where $z = (t_1, \cdots,t_m)$; namely, $R_n(z) = 0$. By Corollary \ref{cor_R_n_hat_R_n_interleave}, w.h.p., $\hat{R}_n(t_1-\epsilon,\cdots,t_m-\epsilon) = 0$; namely, $(t_1-\epsilon,\cdots,t_m-\epsilon)\in\hat{Z}_n$. Thus $z = (t_1,\cdots,t_m)\in\hat{Z}_n+\epsilon$. 

For the second inclusion in (i), let $z\in\hat{Z}_n$, where $z = (t_1,\cdots,t_m)$; namely, $\hat{R}_n(z) = 0$. By Corollary \ref{cor_R_n_hat_R_n_interleave}, w.h.p., $R_n(t_1-\epsilon,\cdots,t_m-\epsilon) = 0$; namely, $(t_1-\epsilon,\cdots,t_m-\epsilon)\in Z_n$. Thus $z = (t_1,\cdots,t_m)\in Z_n+\epsilon$.

To prove (ii), first, let $\nu = \inf\{R_\infty(\vec{t}):\vec{t}\in [0,1]^m\setminus (Z_\infty+\epsilon)\}$; if $[0,1]^m\setminus (Z_\infty+\epsilon)=\varnothing$, simply define $\nu = 1$. Since $Z_\infty$ is the domain where $R_\infty$ takes zero values and $Z_\infty+\epsilon$ is strictly away from $Z_\infty$, by continuity of $R_\infty$ and the fact $R_\infty(1,\cdots,1) = 1$, we must have $\nu> 0$. By Theorem \ref{thm_1st_asym}, w.h.p., $\lVert R_n-R_\infty\lVert_\infty\leq \nu/2$. Thus, by triangle inequality, w.h.p., $R_n\geq \nu-\nu/2 = \nu/2>0$ on $[0,1]^m\setminus(Z_\infty+\epsilon)$. Now, for every $\vec{t}\in Z_n$, $R_n(\vec{t})=0$ and thus $\vec{t}\notin [0,1]^m\setminus Z_\infty+\epsilon$; namely, $\vec{t}\in Z_\infty+\epsilon$. Hence $Z_n\subseteq Z_\infty+\epsilon$.

To prove (iii), let $\vec{t}\in Z_\infty$. Then $R_\infty(\vec{t}) = 0$, where $\vec{t}=(t_1,...,t_m)$; namely $P_K(\bigcap_{i\in [m]}K^{(i)}(t_i))=0$. Since $K^{(i)}(t_i)$ is open for all $i\in [m]$, $\bigcap_{i\in [m]}K^{(i)}(t_i)$ is open with zero measure. Therefore, $\bigcap_{i\in [m]}K^{(i)}(t_i)=\varnothing$ and thus $R_n(\vec{t}) = \frac{1}{n}\cdot\#((\bigcap_{i\in [m]}K^{(i)}(t_i))\cap X_n)=0$. Hence, $\vec{t}\in Z_n$ and $Z_\infty\subseteq Z_n$.
\end{proof}

Immediate from Lemma \ref{lem_Z_n_hat_Z_n_Z_infty_interleave} is 
\begin{corollary}\label{cor_Zero_Interleave}
Let $Z_\infty$ and $\hat{Z}_n$ be defined as in Definition \ref{def:Z_infty_Z_n}. Then, for all $\epsilon>0$, as $n\rightarrow\infty$, w.h.p., 
\[
\hat{Z}_n\subseteq Z_\infty+\epsilon\text{ and } 
Z_\infty\subseteq \hat{Z}_n+\epsilon.
\]
\end{corollary}
We are now able to prove Theorem \ref{thm_inter_con}.
\begin{theorem*}{\bf (Theorem \ref{thm_inter_con}, Interleaving Convergence Theorem)}\\
Let $M_n\in\mathcal{M}_{m,n}^o$ be sampled from a regular pair $\RegPair$. Then, for all $\epsilon>0$, as $n\rightarrow\infty$,
\[
\Pr\left[d_{\INT}(\DDow(\mathcal{F}, P_K),\Dow(S(M_n)))\leq\epsilon\right]\rightarrow 1.
\]
\end{theorem*}
\begin{proof}
Let $\epsilon>0$. We need to prove: as $n\to\infty$, w.h.p.
\[
\DDow(\mathcal{F}, P_K)(t_1-\epsilon,...,t_m-\epsilon)\subseteq\Dow(S(M_n))(t_1,...,t_m)\subseteq\DDow(\mathcal{F}, P_K)(t_1+\epsilon,...,t_m+\epsilon).
\]
By Corollary \ref{cor_Zero_Interleave}, as $n\to\infty$, w.h.p., $\hat{Z}_n\subseteq Z_\infty+\epsilon\text{ and } 
Z_\infty\subseteq \hat{Z}_n+\epsilon$. 

For the first inclusion, let $\sigma\in\DDow(\mathcal{F}, P_K)(t_1-\epsilon,...,t_m-\epsilon)$. Then, by Theorem \ref{thm_Dow_DDow_express_in_hat_R_n_R_infty}, 
\[
R_\infty((t_1-\epsilon)^\sigma,\cdots,(t_m-\epsilon)^\sigma)\neq 0
\]
where $(t_i-\epsilon)^\sigma=(t_i-\epsilon)$ if $i\in\sigma$ and $(t_i-\epsilon)^\sigma=1$ if $i\notin\sigma$. We need to prove $\hat{R}_n(t_1^\sigma,\cdots,t_m^\sigma)\neq 0$. Let us prove by contradiction. Suppose $\hat{R}_n(t_1^\sigma,\cdots,t_m^\sigma)=0$. Then $(t_1^\sigma,\cdots,t_m^\sigma)\in \hat{Z}_n$. Thus $(t_1^\sigma,\cdots,t_m^\sigma)\in Z_\infty+\epsilon$ and $((t_1-\epsilon)^\sigma,\cdots,(t_m-\epsilon)^\sigma)$, meaning $R_\infty(((t_1-\epsilon)^\sigma,\cdots,(t_m-\epsilon)^\sigma))=0$, a contradiction. Thus $\hat{R}_n(t_1^\sigma,\cdots,t_m^\sigma)\neq 0$ and $\sigma\in\Dow(S(M_n))(t_1,...,t_m)$, completing the 1st part. Analogously, the second inclusion can be proved by using Theorem \ref{thm_Dow_DDow_express_in_hat_R_n_R_infty}, with the help of $Z_\infty\subseteq \hat{Z}_n+\epsilon$. 
\end{proof}

\subsection{Proof of Theorem \ref{thm_asym_L_k}, the convergence of $L_k(M_n)$ to $L_k\RegPair$}\label{sec:pf_of_asym_L_k}
In this subsection, we state the well-known Isometry Theorem in topological data analysis and use it to prove Theorem \ref{thm_asym_L_k}. We begin with the definition of a {\it quadrant-tame} persistence module. 
\begin{definition}[Definition 1.12 in \cite{Oudot2015PersistenceT}]
A persistence module $\mathbb{V} = (V_i, v_i^j)$ over $\R$ is {\bf quadrant-tame} if $\rank\ v_i^j<\infty$ for all $i<j$. 
\end{definition}

\begin{theorem}[Isometry Theorem, Theorem 3.1 in \cite{Oudot2015PersistenceT}]\label{thm:isometry_thm}
Let $\mathbb{V}, \mathbb{W}$ be quadrant-tame persistence modules over $\R$. Then
\[
\db(\dgm(\mathbb{V}), \dgm(\mathbb{W})) = \di(\mathbb{V}, \mathbb{W})
\]
where $\db$ is the bottleneck distance between persistence diagrams and $\di$ is the interleaving distance between persistence modules. 
\end{theorem}

Notice that, throughout the paper, all simplicial complexes are subcomplexes of $2^{[m]}$ and hence all vector spaces in the persistence modules we consider are finite dimensional and thus quadrant-tame. Therefore, we have the Isometry Theorem available. In the rest of this section, the proof of Theorem \ref{thm_asym_L_k} is broken into several lemmas based on some newly developed tools. Since the presentation is in logical order instead of the order of ideas, we give a quick overview of how they are pieced together. 

The central observation throughout the proof is Lemma \ref{lem:L_k_sup_expression}, which writes both $L_k\RegPair$ and $L_k(M_n)$ in terms of double supremum expressions. Notice that their expressions only differ in $\DDow\RegPair$ and $\Dow(S(M_n))$, and in $\mathcal{T}\RegPair_+$ and $\mathcal{T}(M_n)_+$, whcih are introduced in Definition \ref{def:diagonal_rays_RegPair_Empirical} and Definition \ref{def:diagonal_rays}. 

With this in mind, it is easy to see that a result that bounds the variation of the double supremum expression when $\DDow\RegPair$ is replaced by $\Dow(S(M_n))$ is needed, which is exactly Lemma \ref{lem:change_complex}. Similarly, a result that bounds the variation of the double supremum expression when $\mathcal{T}\RegPair_+$ is replaced by $\mathcal{T}(M_n)_+$ is also needed, which is Lemma \ref{lem:change_mathcal_T}. We still need to justify the applicability of Lemma \ref{lem:change_complex} and Lemma \ref{lem:change_mathcal_T}, respectively, which corresponds to the Interleaving Convergence Theorem (Theorem \ref{thm_inter_con}) and Lemma \ref{lem:Hausdorff_convergence}, respectively. Now the pieces can be connected and combined to complete the proof of Theorem \ref{thm_asym_L_k}. Notice that Isometry Theorem (Theorem \ref{thm:isometry_thm}) is lurking in the proofs of Lemma \ref{lem:change_complex} and Lemma \ref{lem:change_mathcal_T} and thus playing an important role in the proof of Theorem \ref{thm_asym_L_k}.

\begin{definition}\label{def:diagonal_rays}
Let $\mathcal{T}\subseteq [0,1]^m$. Define the set of {\bf diagonal rays} of $\mathcal{T}$, denoted $\mathcal{T}_+$, by
\begin{equation}
\mathcal{T}_+ \od\{\ray_T:T = (t_1, ..., t_m)\in\mathcal{T}\}
\end{equation}
where $\ray_T\od\left\{(t_1-t, ..., t_m-t):t\in [0,\max_{i\in [m]} t_i]\right\}$. 
\end{definition}

\begin{definition}\label{def:diagonal_rays_RegPair_Empirical}
Let $\RegPair$ be a regular pair and $M_n\in\mathcal{M}_{m,n}^o$ be sampled from $\RegPair$. Define the following two subsets of $[0,1]^m$:
\begin{align*}
& \mathcal{T}\RegPair\od\{(T_1(x), ..., T_m(x)):x\in K\}\text{, and}\\
& \hat{\mathcal{T}}(M_n)\od\left\{\left(\ord_1(M_n,a)/n,\cdots,\ord_m(M_n,a)/n\right): a\in [n]\right\}.
\end{align*}
\end{definition}

Recall that the {\bf Hausdorff distance} between two subsets $S_1, S_2$ of $\R^m$ is defined as
\begin{equation}\label{eq:Hausdordd_distance}
\dH(S_1, S_2) = \inf\{\epsilon>0:S_1\subseteq S_2+B(0,\epsilon),S_2\subseteq S_1+B(0,\epsilon)\}
\end{equation}
where $B(0,\epsilon)$ is the $\epsilon$-ball in $\R^m$ centered at $0$ and $+$ inside the $\inf$ is the operation of Minkowski sum. In the next lemma, we prove that $\hat{\mathcal{T}}(M_n)$ approximates $\mathcal{T}\RegPair$ in Hausdorff distance.
\begin{lemma}\label{lem:Hausdorff_convergence}
Let $M_n\in\mathcal{M}_{m,n}^o$ be  sampled from a regular pair $\RegPair$. Then, as $n\to\infty$, 
\begin{equation}
\dH(\mathcal{T}\RegPair, \hat{\mathcal{T}}(M_n))\to 0 \text{ in probability.}
\end{equation}
\end{lemma}
\begin{proof}
Recall that, for each $i\in [m]$, $T_i = \phi_i\circ f_i$, where $\phi_i$ is a monotone increasing function. Since there is no measure jump in a regular pair (i.e. $P_K(f_i^{-1}(\ell)) = 0$ for all $i\in [m], \ell\in\R$), each $\phi_i$ is continuous and so is each $T_i$. Since each $f_i$ can be extended continuously to $\bar{K}$, we also have each $T_i$ continuously extendable to $\bar{K}$. Since $\bar{K}$ is compact, the function $(T_1, ..., T_m):\bar{K}\to\R^m$ is uniformly continuous.

Let $\epsilon>0$. We need to prove, as $n\to\infty$, w.h.p., 
\begin{equation}
\dH(\mathcal{T}\RegPair, \hat{\mathcal{T}}(M_n))<\epsilon
\end{equation}
By uniform continuity, there exists $\delta>0$ such that, for all $x,y\in K$ with $\lVert x-y\rVert_2\leq\delta$, 
\begin{equation}\label{eq:unif_conti_T_i}
\lVert (T_1(x), \cdots,T_m(x))-(T_1(y), \cdots,T_m(y))\rVert_2\leq\epsilon/2.
\end{equation}
Let $X_n = \{x_1, ..., x_n\}$ be a sample of size $n$, i.i.d. from $\RegPair$. Let us prove that, as $n\to\infty$, w.h.p., 
\begin{equation}\label{claim1:L1.19}
K\subseteq X_n+B(0,\delta)
\end{equation}
Since $K$ is bounded, we may cover $K$ by finitely many small rectangles of diameters smaller than $\delta$, where each rectangle intersects $K$ and the rectangles intersect each other only on their boundaries. Denote the rectangles by $\{R_1, ..., R_N\}$. Let $p_j\od P_K(R_j\cap K)$, which are positive by openness of $K$. Then
\begin{equation}
\Pr[(R_j\cap K)\cap X_n\neq\varnothing,\ \forall\ j\in [N]]\geq 1-\sum_{j = 1}^N (1-p_j)^n.
\end{equation}
Since $N$ is finite and each $1-p_j\in[0,1)$, as $n\to\infty$, w.h.p., $(R_j\cap K)\cap X_n\neq\varnothing$, for all $j = 1, ..., N$. Since the diameter of each $R_j\cap K$ is less than $\delta$, Equation (\ref{claim1:L1.19}) follows.

Let us prove another preparatory result: as $n\to\infty$, w.h.p., 
\begin{equation}\label{eq:ord_i_appx_T_i}
\max_{a\in [n], i\in [m]} \left|\frac{\ord_i(M_n, a)}{n}-T_i(x_a)\right|\leq\sqrt{\frac{\epsilon}{2m}}.
\end{equation}
Treating each $T_i$ as a cumulative distribution function defined on $K$, since $[m]$ is finite, Equation (\ref{eq:ord_i_appx_T_i}) is an immediate consequence of Glivenko-Cantelli Theorem. 

Now we return to the proof. Notice that $\mathcal{T}\RegPair$ is the image of $(T_1, ..., T_m):K\to [0,1]^m$. By Equation (\ref{claim1:L1.19}) and Equation (\ref{eq:unif_conti_T_i}), w.h.p., $\mathcal{T}\RegPair\subseteq (T_1, ..., T_m)(X_n)+B(0,\epsilon/2)$. By Equation (\ref{eq:ord_i_appx_T_i}), $(T_1, ..., T_m)(X_n)\subseteq \hat{\mathcal{T}}(M_n)+B(0,\epsilon/2)$. Thus, $\mathcal{T}\RegPair\subseteq\hat{\mathcal{T}}(M_n)+B(0,\epsilon)$. On the other hand, by Equation (\ref{eq:ord_i_appx_T_i}), $\hat{\mathcal{T}}(M_n)\subseteq\mathcal{T}\RegPair+B(0,\epsilon/2)\subseteq\mathcal{T}\RegPair+B(0,\epsilon)$, completing the proof.
\end{proof}

In the following, we develop the convention of restricting a multi-filtered complex to a diagonal ray as defined in Definition \ref{def:diagonal_rays}. 
\begin{definition}
Let $\mathcal{T}\subseteq [0,1]^m$ and $\mathcal{K} = \{\mathcal{K}(T)=\mathcal{K}(t_1, ..., t_m)\}_{T\in\R^m}$ be a multi-filtered complex indexed over $\R^m$ with $\mathcal{K}(T)\subseteq 2^{[m]}$ for all $T\in\R^m$. For $T = (t_1, ..., t_m)\in\mathcal{T}$, let $\ray_T$ be as in Definition \ref{def:diagonal_rays}. Define the {\bf restriction of $\mathcal{K}$ to $\ray_T$} as the $1$-dimensional filtered complex $\mathcal{K}|_{\ray_T} = \{\mathcal{K}|_{\ray_T}(t)\}_t$, indexed over $t\in [0,\max_{i\in [m]} t_i]$, by 
\begin{equation}
\mathcal{K}|_{\ray_T}(t)\od\mathcal{K}\left(t_1-\max_{i\in [m]} t_i +t,\cdots,t_m-\max_{i\in [m]} t_i+t\right). 
\end{equation}
Since we usually need to use interleaving distance to compare two filtered complexes, we extend the indexing set of $\mathcal{K}|_{\ray_T}$ to $\R$ by 
\begin{equation}
\mathcal{K}|_{\ray_T}(t)\od\begin{cases}
\varnothing & \text{ if $t<0$}\\
2^{[m]} & \text{ if $t>\max_{i\in [m]} t_i$}.
\end{cases}
\end{equation}
\end{definition}

With these conventions, we state the following lemma:
\begin{lemma}\label{lem:L_k_sup_expression}
For each $k\in\{0\}\cup\N$,
\begin{align*}
&L_k\RegPair = \sup_{\ray\in\mathcal{T}\RegPair_+}\sup\{(\beta-\alpha):(\alpha,\beta)\in\dgm(H_k(\DDow\RegPair|_{\ray}))\} \text{, and }\\
&L_k(M_n) = \sup_{\ray\in\hat{\mathcal{T}}(M_n)_+}\sup\{(\beta-\alpha):(\alpha,\beta)\in\dgm(H_k(\Dow(S(M_n))|_{\ray}))\}.
\end{align*}
\end{lemma}
\begin{proof}
The first equality follows from Equation (\ref{eq:L_k_eq2}) in Definition \ref{def_L_k} and the second equality can be obained from Equation (\ref{eq:L_k_emp_eq2}) in Definition \ref{def_L_k_emp}.
\end{proof}
The next lemma will be used to connect Lemma \ref{lem:Hausdorff_convergence} and Lemma \ref{lem:L_k_sup_expression}. 
\begin{lemma}\label{lem:change_mathcal_T}
Let $\mathcal{T}_1,\mathcal{T}_2\subseteq [0,1]^m$ such that $\dH(\mathcal{T}_1,\mathcal{T}_2)<\epsilon$. Let $\mathcal{K} = \{\mathcal{K}(T)\}_{T\in\R^m}$ be a multi-filtered complex indexed over $\R^m$ with $\mathcal{K}(T)\subseteq 2^{[m]}$ for all $T$. Then
\begin{align*}
&\left|\sup_{\ray\in (\mathcal{T}_1)_+}\sup\{b-a:(a,b)\in\dgm(H_k(\mathcal{K}|_\ray))\}-\sup_{\ray\in (\mathcal{T}_2)_+}\sup\{b-a:(a,b)\in\dgm(H_k(\mathcal{K}|_\ray))\}\right|<2\epsilon.
\end{align*}
\end{lemma}
\begin{proof}
For any constant $\eta_1>0$, we may choose $\ray_1\in (\mathcal{T}_1)_+$ such that 
\begin{equation}\label{eq:eta1}
\sup_{\ray\in (\mathcal{T}_1)_+}\sup\{b-a:(a,b)\in\dgm(H_k(\mathcal{K}|_\ray))\}\leq\sup\{b-a:(a,b)\in\dgm(H_k(\mathcal{K}|_{\ray_1}))\}+\eta_1.
\end{equation}
Let $T_1$ be the element in $\mathcal{T}_1$ that $\ray_1$ is constructed from. Since $\dH(\mathcal{T}_1, \mathcal{T}_2)<\epsilon$, there exists $T_2\in\mathcal{T}_2$ such that $\lVert T_1-T_2\rVert_2<\epsilon$. Let $\ray_2\in (\mathcal{T}_2)_+$ be constructed from $T_2$. Then $\dINT(\mathcal{K}|_{\ray_1}, \mathcal{K}|_{\ray_2})<\epsilon$, implying 
\begin{align*}
\di(H_k(\mathcal{K}|_{\ray_1}), H_k(\mathcal{K}|_{\ray_2}))<\epsilon.
\end{align*}
By Isometry Theorem (Theorem \ref{thm:isometry_thm}), 
\begin{equation}\label{eq:bottleneck_on_different_rays}
\db(\dgm(H_k(\mathcal{K}|_{\ray_1})), \dgm(H_k(\mathcal{K}|_{\ray_2})))<\epsilon.
\end{equation}
For any constant $\eta_2>0$, there exists $(a_1, b_1)\in\dgm(H_k(\mathcal{K}|_{\ray_1}))$ such that 
\begin{equation}\label{eq:eta2}
\sup\{b-a:(a,b)\in\dgm(H_k(\mathcal{K}|_{\ray_1}))\}\leq b_1-a_1+\eta_2.
\end{equation}
By Equation (\ref{eq:bottleneck_on_different_rays}), there exists $(a_2, b_2)\in\dgm(H_k(\mathcal{K}|_{\ray_2}))$ such that $\max\{|a_2-a_1|, |b_2-b_1|\}<\epsilon$. Therefore, 
\begin{equation}\label{eq:a_1_2_b_1_2}
b_1-a_1\leq b_2-a_2+2\epsilon.
\end{equation}
Combining Equation (\ref{eq:eta1}), (\ref{eq:eta2}) and (\ref{eq:a_1_2_b_1_2}), we obtain
\begin{align*}
& \sup_{\ray\in (\mathcal{T}_1)_+}\sup\{b-a:(a,b)\in\dgm(H_k(\mathcal{K}|_\ray))\}\\
& \leq b_1-a_1+\eta_1+\eta_2\\
& \leq b_2-a_2+2\epsilon+\eta_1+\eta_2\\
& \leq \sup\{b-a:(a,b)\in\dgm(H_k(\mathcal{K}|_{\ray_2}))\}+2\epsilon+\eta_1+\eta_2\\
& \leq \sup_{\ray\in (\mathcal{T}_2)_+}\sup\{b-a:(a,b)\in\dgm(H_k(\mathcal{K}|_\ray))\}+2\epsilon+\eta_1+\eta_2
\end{align*}
Since $\eta_1, \eta_2>0$ are arbitrary, we obtain
\[
\sup_{\ray\in (\mathcal{T}_1)_+}\sup\{b-a:(a,b)\in\dgm(H_k(\mathcal{K}|_\ray))\}\leq \sup_{\ray\in (\mathcal{T}_2)_+}\sup\{b-a:(a,b)\in\dgm(H_k(\mathcal{K}|_\ray))\}+2\epsilon.
\]
Reversing the role of $\mathcal{T}_1$ and $\mathcal{T}_2$, we may obtain the other direction, completing the proof.
\end{proof}

\begin{lemma}\label{lem:change_complex}
Let $\mathcal{K}$ and $\mathcal{L}$ be multi-filtered complexes indexed over $\R^m$. Let $\mathcal{T}\subseteq [0,1]^m$. If $\dINT(\mathcal{K},\mathcal{L})<\epsilon$, then, for all $k\in\{0\}\cup\N$, 
\begin{align*}
&\left|\sup_{\ray\in \mathcal{T}_+}\sup\{b-a:(a,b)\in\dgm(H_k(\mathcal{K}|_\ray))\}-\sup_{\ray\in \mathcal{T}_+}\sup\{b-a:(a,b)\in\dgm(H_k(\mathcal{L}|_\ray))\}\right|<2\epsilon.
\end{align*}
\end{lemma}
\begin{proof}
For a constant $\eta>0$, let $\ray\in\mathcal{T}_+$ and $(a_0, b_0)\in\dgm(H_k(\mathcal{L}|_\ray))$ such that 
\begin{equation}\label{eq:change_complex_1}
\sup_{\ray\in\mathcal{T}_+}\sup\{b-a:(a,b)\in \dgm(H_k(\mathcal{L}|_\ray))\}\leq b_0-a_0+\eta.
\end{equation}
Consider $\mathcal{K}|_{\ray}$. Since $\dINT(\mathcal{K},\mathcal{L})<\epsilon$, $\dINT(\mathcal{K}|_\ray,\mathcal{L}|_\ray)<\epsilon$. Taking the $H_k$ functor, by the Isometry Theorem, 
\begin{align*}
\db(\dgm(H_k(\mathcal{K}|_\ray)), \dgm(H_k(\mathcal{L}|_\ray)))<\epsilon,
\end{align*}
which implies that there exists $(a_1, b_1)\in\dgm(H_k(\mathcal{K}|_\ray))$ such that $\max(|a_0-a_1|, |b_0-b_1|)<\epsilon$ and hence 
\begin{equation}\label{eq:change_complex_2}
b_0-a_0\leq b_1-a_1+2\epsilon.
\end{equation}
Therefore, 
\begin{align*}
& \sup_{\ray\in\mathcal{T}_+}\sup\{b-a:(a,b)\in\dgm(H_k(\mathcal{L}|_\ray))\}\\
& \leq b_1-a_1+2\epsilon+\eta & \text{by Equation (\ref{eq:change_complex_1}) and (\ref{eq:change_complex_2})}\\
& \leq \sup_{\ray\in\mathcal{T}_+}\sup\{b-a:(a,b)\in\dgm(H_k(\mathcal{K}|_\ray))\}.
\end{align*}
Since $\eta>0$ is arbitrary, we obtain
\begin{equation}
\sup_{\ray\in\mathcal{T}_+}\sup\{b-a:(a,b)\in\dgm(H_k(\mathcal{L}|_\ray))\}\leq \sup_{\ray\in\mathcal{T}_+}\sup\{b-a:(a,b)\in\dgm(H_k(\mathcal{K}|_\ray))\}+2\epsilon.
\end{equation}
Reversing the role of $\mathcal{K}$ and $\mathcal{L}$, the other direction can be otained, completing the proof. 
\end{proof}

With all the above lemmas, we may now present a rigorous proof of Theorem \ref{thm_asym_L_k}. Let us restate Theorem \ref{thm_asym_L_k} for easy reference.
\begin{theorem*}[Theorem \ref{thm_asym_L_k}]
Let $M_n\in\mathcal{M}_{m,n}^o$ be sampled from a regular pair $\RegPair$. Assume that $K$ is bounded and each $f_i$ can be continuously extended to the closure $\bar{K}$. Then, for all $k\in\{0\}\cup\N$, as $n\to\infty$, $L_k(M_n)$ converges to $L_k(\mathcal{F},P_K)$ in probability; namely, for all $\epsilon>0$, 
\[
\lim_{n\to\infty} \Pr\left[\left|L_k(M_n)-L_k(\mathcal{F},P_K)\right|<\epsilon\right]=1.
\]
Moreover, the rate of convergence is independent of $k$. 
\end{theorem*}
\begin{proof}[Proof of Theorem \ref{thm_asym_L_k}]
For notational simplicity, let $\mathcal{T}_1 = \mathcal{T}\RegPair$, $\mathcal{T}_2 = \hat{\mathcal{T}}(M_n)$, $\mathcal{K} = \DDow\RegPair$ and $\mathcal{L} = \Dow(S(M_n))$. Let $\epsilon>0$. By Lemma \ref{lem:L_k_sup_expression}, we need to prove, as $n\to\infty$, w.h.p., 
\begin{align*}
& \left|\sup_{\ray\in(\mathcal{T}_1)_+}\sup\{b-a:(a,b)\in\dgm(H_k(\mathcal{K}|_\ray))\}
-
\sup_{\ray\in (\mathcal{T}_2)_+}\sup\{b-a:(a,b)\in\dgm(H_k(\mathcal{L}|_\ray))\}\right|\leq\epsilon.
\end{align*}
By Interleaving Convergence Theorem (Thoerem \ref{thm_inter_con}), as $n\to\infty$, w.h.p., $\dINT(\mathcal{K}, \mathcal{L})\leq\epsilon/4$. Therefore, by Lemma \ref{lem:change_complex}, 
\begin{equation}\label{eq:Lk_eq1}
\left|\sup_{\ray\in(\mathcal{T}_1)_+}\sup\{b-a:(a,b)\in\dgm(H_k(\mathcal{K}|_\ray))\}
-
\sup_{\ray\in (\mathcal{T}_1)_+}\sup\{b-a:(a,b)\in\dgm(H_k(\mathcal{L}|_\ray))\}\right|\leq \epsilon/2.
\end{equation}
On the other hand, by Lemma \ref{lem:Hausdorff_convergence}, as $n\to\infty$, $\dH(\mathcal{T}_1, \mathcal{T}_2)\leq\epsilon/4$. Therefore, by Lemma \ref{lem:change_mathcal_T},  
\begin{equation}\label{eq:Lk_eq2}
\left|\sup_{\ray\in(\mathcal{T}_1)_+}\sup\{b-a:(a,b)\in\dgm(H_k(\mathcal{L}|_\ray))\}
-
\sup_{\ray\in (\mathcal{T}_2)_+}\sup\{b-a:(a,b)\in\dgm(H_k(\mathcal{L}|_\ray))\}\right|\leq \epsilon/2.
\end{equation}
Hence, combining Equation (\ref{eq:Lk_eq1}) and Equation (\ref{eq:Lk_eq2}), the result follows. 
\end{proof}

\subsection{A lemma used in the proof of Lemma \ref{lem_cent_0_local}}\label{App_Sec:Gap_Lemma4_5}

\begin{lemma}\label{lem:Gap_Lemma4_5}
For $t_1, ..., t_m\in (0,1)$ and $\sigma\subseteq [m]$, if $\bigcap_{i\in [m]} K^{(i)}(t_i)\neq\varnothing$, then there exists $\epsilon>0$ such that $\bigcap_{i\in\sigma} K^{(i)}(t_i-\epsilon)\neq\varnothing$. In addition, by monotonicity of $K^{(i)}(t)$, we also have $\bigcap_{i\in\sigma} K^{(i)}(t_i-\eta)\neq\varnothing$, for all $0<\eta\leq\epsilon$. 
\end{lemma}

\begin{proof}
Let $\epsilon_n$ be a sequence with $\epsilon_n\searrow 0$. Let us first prove that $K^{(i)}(t_i-\epsilon_n)\nearrow K^{(i)}(t_i)$; equivalently, $K^{(i)}(t_i) = \bigcup_{n=1}^\infty K^{(i)}(t_i-\epsilon_n)$. For any $n$, $K^{(i)}(t_i-\epsilon_n)\subseteq K^{(i)}(t_i)$ by definition. Therefore, $\bigcup_{n=1}^\infty K^{(i)}(t_i-\epsilon_n)\subseteq K^{(i)}(t_i)$. For the other inclusion, assume $x\in K^{(i)}(t_i)=f^{-1}(-\infty,\lambda_i(t_i))$. Then $f_i(x)<\lambda_i(t_i)$. By Lemma \ref{lem:ell_t}, $\lambda_i$ is continuous and strictly increasing. Hence, there exists $n$ such that $\lambda_i(t_i-\epsilon_n)>f_i(x)$; in other words, $x\in K^{(i)}(t_i-\epsilon_n)$. Therefore, $x\in\bigcup_{n=1}^\infty K^{(i)}(t_i-\epsilon_n)$, proving the claim. 

Now we have, as $n\nearrow\infty$, $K^{(i)}(t_i-\epsilon_n)\nearrow K^{(i)}(t_i)$. Thus $\bigcap_{i\in\sigma} K^{(i)}(t_i-\epsilon_n)\nearrow\bigcap_{i\in\sigma} K^{(i)}(t_i) $. 
Since $\bigcap_{i\in\sigma} K^{(i)}(t_i)\neq\varnothing$, there must exist $n$ such that $\bigcap_{i\in\sigma} K^{(i)}(t_i-\epsilon_n)\neq\varnothing$. Taking $\epsilon = \epsilon_n$, the result follows.
\end{proof}

\subsection{$\Cent_1$ is open}
This subsection is devoted to the proof of openness of $\Cent_1$. 
\begin{lemma}\label{lem:cent_0_open}
Let $\{f_i:K\to\R\}_{i\in [m]}$ be a collection of quasi-convex $C^1$ functions, where $K$ is open convex in $\R^d$. Then the set $\Cent_1=\left\{x\in K:\cone(\{\nabla f_i(x)\}_{i\in [m]})=\R^d\right\}$ is open in $K$. In particular, $\Cent_1\neq\varnothing$ is equivalent to $P_K(\Cent_1)>0$. 
\end{lemma}
\begin{proof}
Define a function $h:K\times S^{d-1}\to\R$ by $h(x,u) = \max_{i\in [m]} \langle u,\nabla f_i(x)\rangle$. Since each $f_i$ is $C^1$, the functions $(x,u)\mapsto \langle u,\nabla f_i(x)\rangle$ are continuous and hence $h$ is also continuous. For $x\in K$, we define $\rho(x)=\min_{u\in S^{d-1}} h(x,u)$. Let us prove that, for $x\in K$, $\rho(x)>0$ if and only if $\cone(\{\nabla f_i(x)\}_{i\in [m]}) = \R^d$. 

For one direction, let $x\in K$ satisfy $\rho(x)>0$, or, equivalently, $\max_{i\in [m]} \langle u,\nabla f_i(x)\rangle>0$, for all $u\in S^{d-1}$. If $0\in\bd(\conv(\{\nabla f_i(x)\}_{i\in [m]}))$, then the nonzero vector $v$ pointing outward of $\conv(\{\nabla f_i(x)\}_{i\in [m]})$ and orthogonal to the hyperface containing $0$ will make $\max_{i\in [m]} \langle v,\nabla f_i(x)\rangle = 0$, a contradiction. If $0\notin\conv(\{\nabla f_i(x)\}_{i\in [m]})$, then taking $v = -\argmin_{z\in\conv(\{\nabla f_i(x)\}_{i\in [m]})} \langle z, z\rangle$ will make $\max_{i\in [m]} \langle v,\nabla f_i(x)\rangle <0$, also a contradiction. Therefore, $0\in\interior(\conv(\{\nabla f_i(x)\}_{i\in [m]}))$ and hence $\cone(\{\nabla f_i(x)\}_{i\in [m]}) = \R^d$. For the other direction, let $x\in K$ satisfy $\cone(\{\nabla f_i(x)\}_{i\in [m]}) = \R^d$. To prove $\rho(x)>0$, since $u\mapsto \max_{i\in [m]} \langle u,\nabla f_i(x)\rangle$ is continuous and $S^{d-1}$ is compact, it suffices to prove, for all $u\in S^{d-1}$, $\max_{i\in [m]} \langle u, \nabla f_i(x)\rangle >0$. Given $u\in S^{d-1}$, since $\cone(\{\nabla f_i(x)\}_{i\in [m]}) = \R^d$, $u = \sum_{i\in [m]} r_i\cdot\nabla f_i(x)$, for some $r_i\geq 0$. If $\langle u,\nabla f_i(x)\rangle\leq 0$ for all $i\in [m]$, then $\langle u,u\rangle = \sum_{i\in [m]} r_i\langle u,\nabla f_i(x)\rangle\leq 0$, a contradiction. Thus, the other direction is proved and, for $x\in K$, $\rho(x)>0$ if and only if $\cone(\{\nabla f_i(x)\}_{i\in [m]}) = \R^d$. 

Now let $x_0\in K$ such that $\cone(\{\nabla f_i(x_0)\}_{i\in [m]})=\R^d$. By what has been claimed, this is equivalent to $\rho(x_0)>0$. We want to prove that there exists $\epsilon>0$ such that for all $x\in B(x_0,\epsilon)$, $\cone(\{\nabla f_i(x)\}_{i\in [m]})=\R^d$, or equivalently, $\rho(x)>0$. Suppose not, then there exists a sequence $(x_n, u_n)\in K\times S^{d-1}$ such that $x_n\to x_0$ and $h(x_n, u_n)\leq 0$ for all $n$. By compactness of $S^{d-1}$, there is a subsequence $u_{n_j}\to u_0$ and thus by continuity of $h$, $h(x_0, u_0)\leq 0$. However, $h(x_0, u_0)\geq\min_{u\in S^{d-1}} h(x_0,u) = \rho(x_0)>0$, a contradiction. Thus the proof is complete. 
\end{proof}

\subsection{Proof of Theorem \ref{thm_suff_cond_test}}\label{sec:append_proof_suff_con}
Throughout this subsection, $\Cent_0$ is as defined in Definition \ref{def:Cent_0}, $\widehat{\Cent}_0(M_n)$ is as defined in Definition \ref{def_est_cent_1}, and $\hat{Z}_n$ and $Z_\infty$ are as defined in Definition \ref{def:Z_infty_Z_n}. The following two functions play a crucial role throughout the proof of Theorem \ref{thm_suff_cond_test}. 
\begin{definition}\label{def:tau_hat_tau}
Let $X_n=\{x_1, ..., x_n\}$ be sampled from a regular pair $\RegPair$ and $M_n\in\mathcal{M}_{m,n}^o$ be the correspondin data matrix. Define $\tau:K\to [0,1]^m$ and $\hat{\tau}_n:X_n\to [0,1]^m$ by 
\begin{align*}
\tau(x) & \od (T_1(x), ..., T_m(x))\\
\hat{\tau}_n(x_a) & \od \left(\frac{\ord_1(M_n,a)-1}{n}, \cdots, \frac{\ord_m(M_n,a)-1}{n}\right).
\end{align*}
\end{definition}

In order to prove Theorem \ref{thm_suff_cond_test}, we first prove the following lemmas (Lemma \ref{lem:claim0} - Lemma \ref{lem:claim2}). 

\begin{lemma}\label{lem:claim0}
Let $\tau:K\to [0,1]^m$ and $\hat{\tau}_n:X_n\to [0,1]^m$ be defined as in Definition \ref{def:tau_hat_tau}. Then 
\begin{itemize}
\item[(i)] $\tau^{-1}(Z_\infty)=\Cent_0$,
\item[(ii)] $\hat{\tau}_n^{-1}(\hat{Z}_n) = \left \{x_a\in X_n:a\in \widehat{\Cent}_0\left (M_n\right)\right\}$, and
\item[(iii)] $\tau^{-1}(Z_\infty)\cap X_n\subseteq \hat{\tau}_n^{-1}(\hat{Z}_n)$.
\end{itemize}
\end{lemma}
\begin{proof}
To prove (i), 
\begin{align*}
\tau^{-1}(Z_\infty) & = \left\{x\in K: \tau(x)\in Z_\infty\right\} = \left\{x\in K: P_K\left(\bigcap_{i\in [m]} K^{(i)}(T_i(x))\right)= 0\right\}\\
 & =\left\{x\in K: \bigcap_{i\in [m]} K^{(i)}(T_i(x)) = \varnothing\right\}=\left\{x\in K: \bigcap_{i\in [m]} f_i^{-1}(-\infty,f_i(x)) = \varnothing\right\}=\Cent_0.
\end{align*}

To prove (ii), 
\begin{align*}
\hat{\tau}_n^{-1}(\hat{Z}_n) & = \{x_a\in X_n:\hat{\tau}_n(x_a)\in \hat{Z}_n\}\\
 & =
\left\{x_a\in X_n: \bigcap_{i\in [m]}\{b\in [n]:\ord_i(M_n,b)\leq\ord_i(M_n,a)-1\} = \varnothing \right\}\\
 & =
\left\{x_a\in X_n: a\in [n], \bigcap_{i\in [m]} \{b\in [n]:M_{ib}<M_{ia}\}=\varnothing\right\} \\
 & = \{x_a\in X_n:a\in \widehat{\Cent}_0(M_n)\}.
\end{align*}
To prove (iii), assume $x_a\in \tau^{-1}(Z_\infty)\cap X_n$. Then $\bigcap_{i\in [m]} f_i^{-1}(-\infty,f_i(x_a)) = \varnothing$. Thus
\begin{align*}
\hat{R}_n(\hat{i}_n(x_a)) & = \hat{R}_n\left(\frac{\ord_1(M_n,a)-1}{n}, ..., \frac{\ord_m(M_n,a)-1}{n}\right)\\
 & = \frac{1}{n}\cdot\#\left(X_n\cap\left(\bigcap_{i\in [m]} K_n^{(i)}\left(\frac{\ord_i(M_n,a)-1}{n}\right)\right)\right) & & \text{(by Lemma \ref{lem_hat_R_n_expression})}\\
 & = \frac{1}{n}\cdot\#\left(X_n\cap\left(\bigcap_{i\in [m]} f_i^{-1}(-\infty,f_i(x_a))\right)\right) & & \text{(by Definition \ref{def_K_n^(i)(t)})}\\
 & =0. 
\end{align*}
Hence, $x_a\in\hat{\tau}_n^{-1}(\hat{Z}_n)$ and the inclusion in (iii) follows. 
\end{proof}

\begin{lemma}\label{lem:claim1}
Let $\tau:K\to [0,1]^m$ and $\hat{\tau}_n:X_n\to [0,1]^m$ be defined as in Definition \ref{def:tau_hat_tau}. Then, for any $\delta>0$, as $n\to\infty$, w.h.p., $\hat{\tau}_n^{-1}(\hat{Z}_n)\subseteq \tau^{-1}(Z_\infty+\delta)$.\footnote{See Equation (\ref{eq:Z_plus_epsilon}) for the definition of $Z_\infty+\delta$.}
\end{lemma}

\begin{proof}
By Corollary \ref{cor_Zero_Interleave}, w.h.p., 
\begin{equation}\label{emp-zero-approx}
\hat{Z}_n\subseteq Z_\infty+\delta/2.
\end{equation}
If $Z_\infty = [0,1]^m$, then $\tau^{-1}(Z_\infty+\delta) = K$ and hence $\hat{\tau}_n^{-1}(\hat{Z}_n)\subseteq \tau^{-1}(Z_\infty+\delta)$ clearly holds. Now suppose $Z_\infty\neq [0,1]^m$. Since $Z_\infty$ is closed, $int([0,1]^m\setminus Z_\infty)$ is nonempty and open, say containing $x_0$. Choose $\epsilon>0$ such that $\epsilon<\delta$ and $B(x_0,\sqrt{d}\epsilon)\subseteq int([0,1]^m\setminus Z_\infty)$. Then $Z_\infty+\epsilon/2\subsetneqq Z_\infty+\epsilon$. Thus $\mu_* \od d_H(Z_\infty+\epsilon/2, Z_\infty+\epsilon)/\sqrt{d}$ is a positive number, where $d_H$ is the Hausdorff distance.\footnote{See Equation (\ref{eq:Hausdordd_distance}) for the definition of Hausdorff distance.} Since $Z_\infty+\epsilon/2+\epsilon/2 = Z_\infty+\epsilon$, we have the inequality
\begin{equation}\label{Haus-diff}
\mu_*\leq\sqrt{d\cdot(\epsilon/2)^2}/\sqrt{d} = (\epsilon/2)<\delta/2.
\end{equation}
By Equation (\ref{eq:ord_i_appx_T_i}) and definition of $\tau$ and $\hat{\tau}_n$, w.h.p., 
\begin{equation}\label{unif-approx}
\sup_{x\in X_n}\lVert \tau(x)-\hat{\tau}_n(x)\rVert_2 <\mu_*/2.
\end{equation}
There is one more property of $Z_\infty$, following from monotonicity of $R_\infty$, that we need in this proof: if $(t_1, ..., t_m)\in Z_\infty$ and $(t_1', ..., t_m')\leq (t_1, ..., t_m)$, then $(t_1', ..., t_m')\in Z_\infty$; in short, {\it $Z_\infty$ is closed under $\leq$}. 

Now we can prove the inclusion. Assume $x\in \hat{\tau}_n^{-1}(\hat{Z}_n)$, namely, $\hat{\tau}_n(x)\in\hat{Z}_n$. Then
\[
\begin{matrix*}[l]
\tau(x)&\in\hat{\tau}_n(x)+\overline{B(0,\mu_*/2})& \text{(by (\ref{unif-approx}))} \\
	&\subseteq \hat{Z}_n+\overline{B(0,\mu_*/2})& \text{(since $\hat{\tau}_n(x)\in\hat{Z}_n$)} \\
    &\subseteq Z_\infty+\delta/2+\overline{B(0,\mu_*/2})& \text{(by (\ref{emp-zero-approx}))}\\
    &\subseteq Z_\infty+\delta/2+\mu_*/2& \text{(since $Z_\infty$ is closed under $\leq$)}\\
    &\subseteq Z_\infty+\delta/2+\delta/4& \text{(by (\ref{Haus-diff}))}\\
    &\subseteq Z_\infty+\delta.& 
\end{matrix*}
\]
Therefore, $x\in \tau^{-1}(Z_\infty+\delta)$. Since $x\in \hat{\tau}_n^{-1}(\hat{Z}_n)$ is arbitrary, the proof is complete.
\end{proof}

\begin{lemma}\label{lem:claim2}
Let $X_n\subset K$ be a point cloud of size $n$, sampled from a regular pair $\RegPair$. Let $\tau$ be defined as in Definition \ref{def:tau_hat_tau}. Then, for all $\epsilon>0$, there exists $\delta_0>0$ such that, as $n\to\infty$, w.h.p.,
\[
\left|\frac{\#(X_n\cap \tau^{-1}(Z_\infty+\delta_0))}{n}-P_K(\Cent_0)\right|<\epsilon/2.
\]
\end{lemma}
\begin{proof}
Let us first prove that $\lim_{\delta\searrow 0} P_K(\tau^{-1}(Z_\infty+\delta)) = P_K(\Cent_0)$. Since $Z_\infty$ is closed, $Z_\infty+\delta\searrow Z_\infty$ as $\delta\searrow 0$. Thus, $\tau^{-1}(Z_\infty+\delta)\searrow \tau^{-1}(Z_\infty)$. Therefore, by monotone convergence theorem, 
$\lim_{\delta\searrow 0} P_K(\tau^{-1}(Z_\infty+\delta))= P_K(\tau^{-1}(Z_\infty)) = P_K(\Cent_0)$.

Since $\lim_{\delta\searrow 0} P_K(\tau^{-1}(Z_\infty+\delta)) = P_K(\Cent_0)$, we may choose $\delta_0$ such that 
\begin{equation}\label{eq:claim2_part1}
|P_K(\tau^{-1}(Z_\infty+\delta_0))-P_K(\Cent_0)|<\epsilon/4.
\end{equation}
Note that $X_n\cap \tau^{-1}(Z_\infty+\delta)$ is an i.i.d. sample of $\tau^{-1}(Z_\infty+\delta)\subseteq K$. Thus, by law of large numbers, as $n\to\infty$, w.h.p., 
\begin{equation}\label{eq:claim2_part2}
\left|\frac{\#(X_n\cap \tau^{-1}(Z_\infty+\delta_0))}{n}-P_K(\tau^{-1}(Z_\infty+\delta_0))\right|<\epsilon/4.
\end{equation}
Combining Equation (\ref{eq:claim2_part1}) and (\ref{eq:claim2_part2}), the result follows. 
\end{proof}

With the help of previous lemmas, we give a proof of Theorem \ref{thm_suff_cond_test}.

\begin{proof}[Proof of Theorem \ref{thm_suff_cond_test}]
Let $\tau$ and $\hat{\tau}_n$ be defined as in Definition \ref{def:tau_hat_tau}. Let $\epsilon>0$. Note that $X_n\cap \tau^{-1}(Z_\infty)$ is an i.i.d. sample of $\tau^{-1}(Z_\infty)=\Cent_0\subseteq K$. Thus, by law of large numbers, as $n\to\infty$, w.h.p., $\left|\#(X_n\cap \tau^{-1}(Z_\infty))/n-P_K(\Cent_0)\right|<\epsilon/2$. Therefore, as $n\to\infty$, w.h.p.,
\[
\begin{matrix*}[l]
P_K(\Cent_0)-\epsilon/2 & <\#(X_n\cap \tau^{-1}(Z_\infty))/n & \\
              &\leq\#(\hat{\tau}_n^{-1}(\hat{Z}_n))/n & & \text{(by (iii) of Lemma \ref{lem:claim0})}\\
              &=\#(\widehat{\Cent}_0(M_n))/n & & \text{(by Lemma \ref{lem:claim0})}
\end{matrix*}
\]
On the other hand, as $n\to\infty$, w.h.p., 
\[
\begin{matrix*}[l]
\frac{\#(\widehat{\Cent}_0(M_n))}{n} & = \#(\hat{\tau}_n^{-1}(\hat{Z}_n))/n & & \text{(by Lemma \ref{lem:claim0})}\\
			& = \#(\hat{\tau}_n^{-1}(\hat{Z}_n)\cap X_n)/n & & \text{(since $\hat{\tau}_n^{-1}(\hat{Z}_n)\subseteq X_n$)}\\
            & \leq  \#(\tau^{-1}(Z_\infty+\delta_0)\cap X_n)/n & & \text{(by Lemma \ref{lem:claim1})}\\
            & \leq P_K(\Cent_0)+\epsilon/2& & \text{(by Lemma \ref{lem:claim2})}.  
\end{matrix*}
\]
Therefore, as $n\to\infty$, w.h.p., $\left|\frac{\#(\widehat{\Cent}_0(M_n))}{n}-P_K(\Cent_0)\right|\leq\epsilon/2$. 

\end{proof}

%%%%%%%
\section*{Acknowledgments}
\noindent This work was  supported by the NSF IOS-155925   grant.  

%%%%%%%%%%%%%%%%%%%%%%%%%%%%%%%%%%%%%%%%%%%%%%%%%%%%%

\bibliography{mybib}
\bibliographystyle{amsplain}

\end{spacing}
\end{document}